\documentclass[11pt]{article}

\usepackage{amssymb,amsfonts,amsmath,amsthm}
\usepackage[top=1.0in, bottom=1.0in, left=1.0in, right=1.0in]{geometry}
\usepackage{color,cite,enumerate,graphicx,soul,marginnote}

\usepackage{authblk}

\def\Xint#1{\mathchoice
   {\XXint\displaystyle\textstyle{#1}}%
   {\XXint\textstyle\scriptstyle{#1}}%
   {\XXint\scriptstyle\scriptscriptstyle{#1}}%
   {\XXint\scriptscriptstyle\scriptscriptstyle{#1}}%
   \!\int}
\def\XXint#1#2#3{{\setbox0=\hbox{$#1{#2#3}{\int}$}
     \vcenter{\hbox{$#2#3$}}\kern-.5\wd0}}

\def\dashint{\Xint-}

\newcommand{\eps}{\varepsilon}
\newcommand{\R}{\mathbb R}

\newcommand{\m}{\mathbf{m}}
\newcommand{\brackets}[1]{\left( #1 \right)}

\newtheorem{theorem}{Theorem}
\newtheorem{proposition}[theorem]{Proposition}

\newtheorem{lemma}[theorem]{Lemma}

\newtheorem{remark}[theorem]{Remark}

\title{\bf One-dimensional in-plane edge domain walls in ultrathin
  ferromagnetic films}

\author[1]{Ross G. Lund} 
\author[1]{Cyrill B. Muratov}
\author[2]{Valeriy V. Slastikov}

\affil[1]{\em Department of Mathematical Sciences, New Jersey
    Institute of Technology, Newark, NJ 07102, USA} 
\affil[2]{\em School of Mathematics, University of Bristol,
    Bristol BS8 1TW, UK}

\date{\today}

\begin{document}

\maketitle

\begin{abstract}
  We study existence and properties of one-dimensional edge domain
  walls in ultrathin ferromagnetic films with uniaxial in-plane
  magnetic anisotropy. In these materials, the magnetization vector is
  constrained to lie entirely in the film plane, with the preferred
  directions dictated by the magnetocrystalline easy axis. We consider
  magnetization profiles in the vicinity of a straight film edge
  oriented at an arbitrary angle with respect to the easy axis. To
  minimize the micromagnetic energy, these profiles form transition
  layers in which the magnetization vector rotates away from the
  direction of the easy axis to align with the film edge.  We prove
  existence of edge domain walls as minimizers of the appropriate
  one-dimensional micromagnetic energy functional and show that they
  are classical solutions of the associated Euler-Lagrange equation
  with Dirichlet boundary condition at the edge. We also perform a
  numerical study of these one-dimensional domain walls and uncover
  further properties of these domain wall profiles.
\end{abstract}

\section{Introduction}
\label{sec:introduction}

The field of ferromagnetism of thin films is currently undergoing a
renaissance driven by advances in theory, experiment and technology
\cite{dennis02,desimone06r,
  braun12,vonbergmann14,stamps14,kent15,manipatruni16}. Study of
magnetic domain walls is remaining at the forefront of this activity
and attracts a lot of attention from engineering, physical and
mathematical communities. These studies are in part motivated by a new
field of applied physics -- {\it spintronics} -- offering a great
promise for creating the next generation of data storage and logic
devices combining spin-dependent effects with conventional
charge-based electronics \cite{allwood05, bader10}.

There are two most common types of domain walls connecting the
distinct preferred directions of magnetization in uniaxial materials:
{\it Bloch and N\'eel walls} \cite{hubert}. Bloch walls appear in bulk
ferromagnets, where the magnetization profile connecting the two
opposite easy axis directions prefers an out-of-plane rotation with
respect to the plane spanned by the wall direction and the easy
axis. In ultrathin ferromagnetic films, on the other hand, the stray
field energy penalizes out-of-plane rotations, and as a result the
magnetization profile is constrained to the film plane. A domain wall
profile connecting the two distinct preferred directions of
magnetization via an in-plane rotation is called a {\it N\'eel wall}.

N\'eel walls have been thoroughly investigated theoretically since
their discovery, and their internal structure is currently fairly well
understood. The main characteristics of the one-dimensional N\'eel
wall profile typically include an inner core, logarithmically decaying
intermediate regions and algebraic tails. These features have been
predicted theoretically, using micromagnetic treatments \cite{hubert,
  dietze61, riedel71, garcia99, mo:jcp06}, and verified experimentally
\cite{berger92}.  Recent rigorous mathematical studies of N\'eel walls
confirmed these predictions and provided more refined information
about the profile of the N\'eel wall, including uniqueness,
regularity, monotonicity, symmetry, stability and precise rate of
decay \cite{melcher03, garcia04, desimone06, capella07,
  cm:non13,my:prsla16}.

Another type of a domain wall has been recently observed in ultrathin
ferromagnetic films with perpendicular anisotropy and strong
antisymmetric exchange referred to as Dzyaloshinskii-Moriya
interaction (DMI). The presence of DMI significantly alters the
structure of domain walls, leading to formation of {\it chiral domain
  walls} in the interior and {\it chiral edge domain walls} at the
boundary of the ferromagnetic sample \cite{thiaville12, rohart13,
  ms:prsla16}. These chiral domain walls and chiral edge domain walls
play a crucial role in producing new types of magnetization patterns
inside a ferromagnet and have been rigorously analyzed in
\cite{ms:prsla16}.

It is well known that magnetization configurations in ferromagnets are
significantly affected by the presence of material boundaries
\cite{hubert, dennis02, desimone06r}. To reduce the stray field, the
magnetization vector tries to stay tangential to the material
boundary, thus minimizing the presence of boundary magnetic
charges. In ultrathin films, this forces the magnetization vector to
lie almost entirely in the film plane and align tangentially along the
film's lateral edges \cite{kohn05arma}. At the same time, these
geometrically preferred directions may disagree with the intrinsic
directions in the bulk film, determined by either a strong in-plane
uniaxial crystalline anisotropy or an external in-plane magnetic
field. The result of this incompatibility is another type of magnetic
domain walls -- {\it edge domain walls}. These domain walls have been
observed experimentally in magnetically coupled bilayers in the shape
of strips with an easy axis normal to the strip \cite{ruhrig90}, and
in single-layer strips with negligible crystalline anisotropy and
varying in-plane uniform magnetic field \cite{mattheis97}.  

The origin of edge domain walls is due to the competition between
magnetostatic, exchange and anisotropy energies. Strong uniaxial
anisotropy defines the two preferred magnetization directions within
the ferromagnetic film plane. On the other hand, at the film edge the
magnetostatic energy penalizes the magnetization component normal to
the edge, and consequently the magnetization prefers to lie in-plane
and tangentially to the edge of the ferromagnetic film. The exchange
energy allows for a continuous transition between these states and as
a result an edge domain wall connecting the direction tangent to the
film edge and the anisotropy easy axis direction is created.

Experimental observations in soft ferromagnetic thin films and
bilayers \cite{mattheis97,ruhrig90} indicate that edge domain walls,
formed near the boundary of the sample due to a misalignement of the
tangential and applied field directions, have an essentially
one-dimensional character. This is confirmed by micromagnetic
simulations in extended ferromagnetic strips performed in several
regimes, including strong in-plane uniaxial anisotropy with no applied
field (see Fig.~\ref{f:strip2d}) and no crystalline anisotropy with
strong in-plane applied field (results not shown). Numerical
simulations suggest that away from the side edges the domain walls
have essentially one-dimensional profiles. Therefore, in order to
investigate these profiles it is enough to model their behavior,
employing a simplified one-dimensional micromagnetic energy capturing
the essential features of the wall profiles. Such a description is
expected to be appropriate for strips of soft ferromagnetic materials
whose thickness does not exceed significantly the exchange length and
whose width is much larger than the N\'eel wall width.

The goal of this paper is to understand the formation of edge domain
walls viewed as global energy minimizers of a reduced one-dimensional
micromagnetic energy. We begin our analysis by deriving a
one-dimensional energy functional describing edge domain walls (see
\eqref{Ethb2}). Since we are specifically interested in
one-dimensional domain wall profiles, we consider the problem on an
unbounded domain consisting of a ferromagnetic film occupying a
half-plane times a fixed interval with small thickness. However, this
setup makes the energy of the wall infinite due to inconsistency
between the preferred magnetization directions at the film edge and
inside the film (see section~\ref{sec:statement-results}). Therefore,
in order to have a well defined minimization problem, we need to
renormalize the one-dimensional energy per unit edge length in a
suitable way (see \eqref{Eb}). We show existence of a minimizer for
this energy, using standard methods of the calculus of variations; see
Theorem~\ref{t:exist}. The main difficulty lies in dealing with
nonlocal magnetostatic energy term and identifying the proper space
where the minimization problem makes sense.
\begin{figure}
  \centering
  \centering \includegraphics[width=6.5in]{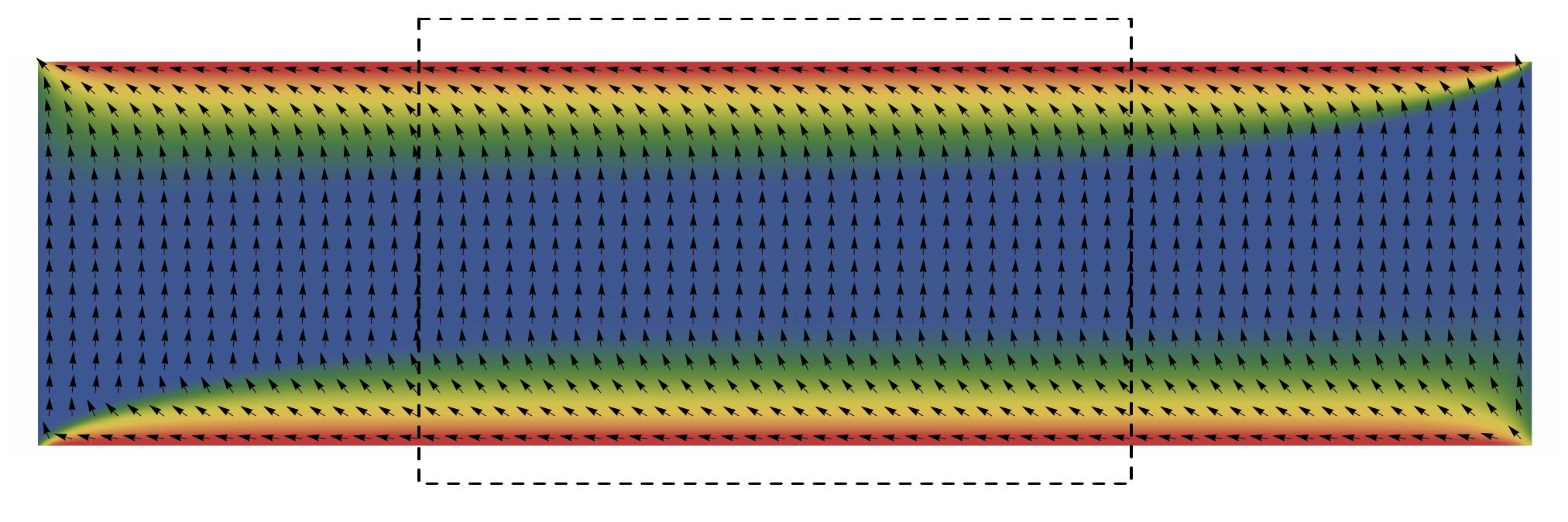}
  \caption{A magnetization configuration containing edge domain walls
    in a strip obtained from micromagnetic simulations of a
    $20.7 \mu$m$\times 5.2\mu$m$\times 4$nm permalloy sample with
    vertical uniaxial anisotropy and no applied field (for further
    details, see section~\ref{sec:num}). The colormap corresponds to
    the angle between the magnetization vector (also shown by arrows)
    and the $y$-axis. Inside the dashed box (i.e., far from the side
    edges) the edge wall profiles are essentially one-dimensional.}
  \label{f:strip2d}
\end{figure}
We continue our analysis by deriving the Euler-Lagrange equation
characterizing the profile of the edge domain wall; see Theorem
\ref{t:regular}.  This seemingly straightforward task, however,
requires a rather careful and proper dealing with the nonlocal energy
term. The main difficulty is related to the fact that we only have
rather limited regularity of the energy minimizing solutions a
priori. The information about further regularity is usually recovered
through the use of the Euler-Lagrange equation and a bootstrap
argument. As this information is not yet available, we need to
carefully analyze the nonlocal term using methods from fractional
Sobolev spaces and recover a weak form of the Euler-Lagrange equation.
After deriving the Euler-Lagrange equation, we can prove higher
regularity of the solutions using an adaptation of the standard
elliptic regularity techniques. However, due to the difficulties
arising in dealing with nonlocality we can only show the $C^2$
regularity of solutions. Further application of the bootstrap argument
is then hindered by the lack of integrability of the contribution of
the nonlocal term to the Euler-Lagrange equation when higher
derivatives of the solution are considered, and the second derivative
of the solution indeed blows up at the film edge.

After establishing existence and regularity of the edge domain wall
profiles, we investigate two specific regimes where we can provide
refined information about the properties of the energy minimizing
solutions of the Euler-Lagrange equation; see Theorem \ref{t:gammanu}
and Theorem \ref{t:gammabeta}. The first regime that we consider is
the regime of relatively small magnetostatic energy, which corresponds
to very thin films. In this regime we show that all minimizers of the
energy \eqref{Eb} are close to the standard local N\'eel wall-type
profile. The second regime is the regime in which the boundary tangent
and the easy axis directions are nearly parallel. In this case we show
that there is a unique minimizer of the energy in \eqref{Eb}, and this
minimizer is close to the uniform state. We corroborate our analytical
findings and provide more information about the profiles of edge
domain walls, using one-dimensional numerical simulations that employ
the method from \cite{mo:jcp06}.

Our paper is organized as follows. In section~\ref{sec:model},
starting from the full three-dimensional micromagnetic model we derive
a variational model for edge domain walls that we intend to
investigate in this paper. Section~\ref{sec:statement-results} is
devoted to a rigorous formulation of the problem and includes the
statements of the main results about existence, regularity and the
qualitative features of edge domain walls. In
sections~\ref{sec:proof-theor-reft}, \ref{sec:proof-theor-reft1},
\ref{sec:proof-theor-gammanu} and \ref{sec:proof-theor-gammabeta} we
prove the main theorems formulated before in
section~\ref{sec:statement-results}. In section~\ref{sec:num}, we
present the results of numerical simulations, compare them with our
analytical findings and discuss open problems. Finally, in
Appendix~\ref{append} we provide a rigorous derivation of the
one-dimensional micromagnetic energy in magnetic strips under natural
assumptions on the one-dimensional magnetization profile.

\section{Model}
\label{sec:model}

Consider a uniaxial ferromagnet occupying a domain
$\Omega \subset \R^3$, with the easy axis oriented along the second
coordinate direction. Then the micromagnetic energy associated with
the magnetization state of the sample reads, in the SI units
\cite{hubert,landau8}:
\begin{multline}
  \label{Ephys}
  E(\mathbf M) = {A \over M_s^2} \int_\Omega |\nabla \mathbf M|^2 \,
  d^3 r + {K \over M_s^2} \int_\Omega (M_1^2 + M_3^2) d^3 r \\
  - \mu_0 \int_\Omega \mathbf M \cdot \mathbf H \, d^3 r + \mu_0
  \int_{\R^3} \int_{\R^3} {\nabla \cdot \mathbf M(\mathbf r) \, \nabla
    \cdot \mathbf M(\mathbf r') \over 8 \pi | \mathbf r - \mathbf r'|}
  \, d^3 r \, d^3 r'.
\end{multline}
Here $\mathbf M = (M_1, M_2, M_3)$ is the magnetization vector that
satisfies $|\mathbf M|=M_s$ in $\Omega$ and $\mathbf M = 0$ in
$\R^3\setminus \Omega$, the positive constants $M_s$, $A$ and $K$ are
the saturation magnetization, exchange constant and the anisotropy
constant, respectively, $\mathbf H$ is an applied external field, and
$\mu_0$ is the permeability of vacuum. In \eqref{Ephys}, the terms in
the order of appearance are the exchange, crystalline anisotropy,
Zeeman and stray field terms, respectively, and
$\nabla \cdot \mathbf M$ is understood distributionally.

In this paper, we are interested in the situation in which $\Omega$ is
a flat ultra-thin film domain, i.e., we have
$\Omega = D \times (0, d)$, where $D \subset \R^2$ is a planar domain
specifying the film shape and $d$ is the film thickness of a few
nanometers. In this case the magnetization is expected to be
essentially independent from the third coordinate, and the full
three-dimensional micromagnetic energy admits a reduction to an energy
functional that depends only on the average of the magnetization over
the film thickness (see, e.g., \cite[Lemma 3]{kohn05arma}; for an
analytical treatment in a closely related context, see
\cite{kmn:arma,m:cmp}). Therefore, we introduce an ansatz
$\mathbf M(x_1, x_2, x_3) = M_s (\mathbf m(x_1, x_2), 0)
\chi_{(0,d)}(x_3)$, where $\mathbf m : \R^2 \to \R^2$ is a
two-dimensional in-plane magnetization vector satisfying
$|\mathbf m| = 1$ in $D$ and $|\mathbf m| = 0$ outside $D$, and
$\chi_{(0,d)}$ is the characteristic function of $(0,d)$.  Next, we
define the exchange length $\ell$, the Bloch wall width $L$, and the
{\em thin film parameter} $\nu$ measuring the relative strength of the
magnetostatic energy \cite{mo:jcp06}:
\begin{align}
  \label{lLQ}
  \ell = \sqrt{2 A \over \mu_0 M_s^2}, \qquad L = \sqrt{A \over K},
  \qquad \nu =  {\mu_0 M_s^2 d \over 2 \sqrt{A K}},
\end{align}
and note that the above ansatz is relevant when $d \lesssim \ell$
\cite{desimone00,kohn05arma,mo:jcp06,doering14,doering16}.  Then,
measuring the energy in the units of $2 A d$ and lengths in the units
of $L$, we obtain the following expression for the energy as a
function of $\mathbf m$ \cite{garcia99}:
\begin{align}
  \label{Em}
  E(\mathbf m) = \frac12 \int_D \left( |\nabla \mathbf m|^2 + m_1^2
  - 2 \mathbf h \cdot \mathbf m \right) d^2 r + {\nu \over 2}
  \int_{\R^2} \int_{\R^2} 
  K_\delta(|\mathbf r - \mathbf r'|) \nabla \cdot \mathbf m(\mathbf
  r) \, \nabla \cdot \mathbf m(\mathbf
  r) \, d^2 r \, d^2 r',
\end{align}
where $\delta = d / L$ is the dimensionless film thickness,
\begin{align}
  \label{Kd}
  K_\delta(r) = {1 \over 2 \pi \delta} \left\{ \ln \left( {\delta +
  \sqrt{\delta^2 + r^2} \over  r} \right) -
  \sqrt{1 + { r^2 \over \delta^2}} + { r \over
  \delta} \right\},
\end{align}
and we set $\mathbf H = K / (\mu_0 M_s) (\mathbf h, 0)$ for
$\mathbf h : \R^2 \to \R^2$, assuming that the applied field lies in
the film plane. More explicitly, assuming that $\partial D$ is of
class $C^2$, we have
\begin{multline}
  \label{Emm}
  E(\mathbf m) = \frac12 \int_D \left( |\nabla \mathbf m|^2 + m_1^2 -
    2 \mathbf h \cdot \mathbf m \right) d^2 r + {\nu \over 2} \int_D
  \int_D K_\delta(|\mathbf r - \mathbf r'|) \nabla \cdot \mathbf
  m(\mathbf r) \, \nabla \cdot \mathbf m(\mathbf r) \, d^2 r \, d^2
  r' \\
  - \nu \int_D \int_{\partial D} K_\delta(|\mathbf r - \mathbf r'|)
  \nabla \cdot \mathbf m(\mathbf r) (\mathbf m(\mathbf r') \cdot
  \mathbf n(\mathbf r')) \, d \mathcal H^1(\mathbf r') \, d^2 r \\
  + {\nu \over 2} \int_{\partial D} \int_{\partial D}
  K_\delta(|\mathbf r - \mathbf r'|) (\mathbf m(\mathbf r) \cdot
  \mathbf n(\mathbf r)) (\mathbf m(\mathbf r') \cdot \mathbf n(\mathbf
  r')) \, d \mathcal H^1(\mathbf r') \, d \mathcal H^1(\mathbf r),
\end{multline}
where $\mathbf n$ is the outward unit normal vector to $\partial D$,
and we took into account that the distributional divergence of
$\mathbf m$ is the sum of the absolutely continuous part in $D$ and a
jump part on $\partial D$.

We now consider the thin film limit introduced in \cite{mo:jcp06} by
sending $\delta$ to zero with $\nu$ and $D$ fixed. Observe that when
$\delta$ is small, we have
\begin{align}
  \label{Kdd}
  K_\delta(r) \simeq {1 \over  4 \pi r} \qquad
  \text{and} \qquad
  \int_{\partial D} K_\delta(|\mathbf r - \mathbf r'|) \, d
  \mathcal H^1(\mathbf r') \simeq {1 \over 2 \pi} \ln \delta^{-1}.
\end{align}
Therefore, when $\mathbf m$ does not vary appreciably on the scale of
$\delta$, to the leading order we have $E(\mathbf m) \simeq
E_\delta(\mathbf m)$, where
\begin{multline}
  \label{Emd}
  E_\delta(\mathbf m) = \frac12 \int_D \left( |\nabla \mathbf m|^2 +
    m_1^2 - 2 \mathbf h \cdot \mathbf m \right) d^2 r + {\nu \over 8
    \pi} \int_D \int_D {\nabla \cdot \mathbf m(\mathbf r) \, \nabla
    \cdot \mathbf m(\mathbf r) \over |\mathbf r - \mathbf r'|} \, d^2
  r \, d^2
  r' \\
  - {\nu \over 4 \pi} \int_D \int_{\partial D} {\nabla \cdot \mathbf
    m(\mathbf r) (\mathbf m(\mathbf r') \cdot \mathbf n(\mathbf r'))
    \over |\mathbf r - \mathbf r'|} \, d \mathcal H^1(\mathbf r') \,
  d^2 r + {\nu \ln \delta^{-1} \over 4 \pi} \int_{\partial D} (\mathbf
  m(\mathbf r) \cdot \mathbf n(\mathbf r))^2 \, d \mathcal H^1(\mathbf
  r).
\end{multline}
Since the last term in \eqref{Emd} blows up as $\delta \to 0$, unless
$\mathbf m \cdot \mathbf n = 0$ a.e. on $\partial D$, in the limit we
recover 
\begin{align}
  \label{E0}
  E_0(\mathbf m) = \frac12 \int_D \left( |\nabla \mathbf m|^2 +
    m_1^2 - 2 \mathbf h \cdot \mathbf m \right) d^2 r + {\nu \over 8
    \pi} \int_D \int_D {\nabla \cdot \mathbf m(\mathbf r) \, \nabla
    \cdot \mathbf m(\mathbf r) \over |\mathbf r - \mathbf r'|} \, d^2
  r \, d^2
  r',
\end{align}
with admissible configurations $\mathbf m \in H^1(D; \mathbb S^1)$
satisfying Dirichlet boundary condition $\mathbf m = s \mathbf t$ on
$\partial D$, where $\mathbf t$ is the positively oriented unit
tangent vector to $\partial D$ and $s : \partial D \to \{-1,1\}$. In
fact, since the trace of $\mathbf m$ belongs to
$H^{1/2}(\partial D; \R^2)$, the function $s$ is necessarily constant
on each connected component of $\partial D$. Note that this creates a
topological obstruction in the case when $D$ is simply connected,
giving rise to boundary vortices at the level of $E_\delta$
\cite{moser04,kohn05arma,kurzke06}. At the same time, it is clear that
for suitable multiply connected domains the considered admissible
class is non-empty. A canonical example of the latter is an annulus
(for a physics overview, see \cite{klaui03}). In the absence of
crystalline anisotropy and applied field, the ground state of the
magnetization in an annulus is easily seen to be a vortex
state. However, this result no longer holds in the presence of
crystalline anisotropy, since the latter does not favor alignment of
$\mathbf m$ with the boundaries. In large annuli, this would lead to
the formation of a boundary layer, in which the magnetization rotates
from the direction tangential to the boundary to the direction of the
easy axis. We call such magnetization configurations {\em edge domain
  walls} (for similar objects in a different micromagnetic context,
see \cite{ms:prsla16}).

\begin{figure}
  \centering
  \centering \includegraphics[width=3in]{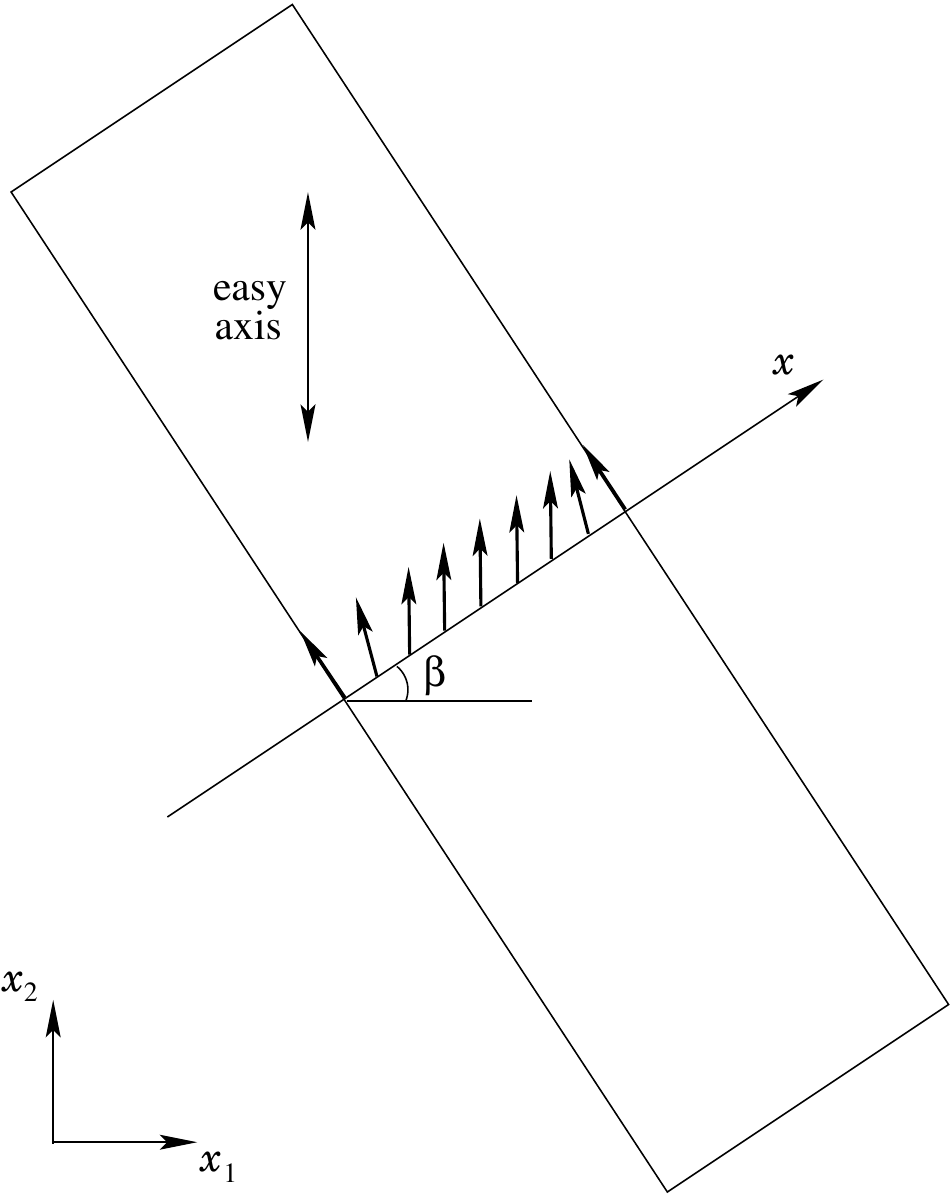}
  \caption{Illustration of the strip geometry.}
  \label{f:strip}
\end{figure}

Focusing on one-dimensional transition profiles in the vicinity of the
boundary, we now consider $D$ to be a strip of width $w$ oriented at
an angle $\beta \in [0, \pi/2]$ with respect to the easy axis (see
Fig. \ref{f:strip}). We define
\begin{align}
  \label{x}
  x = x_1 \cos \beta + x_2 \sin \beta
\end{align}
to be the variable in the direction normal to the strip axis. Then,
with the applied field $\mathbf h$ set to zero, the energy of a
magnetization configuration $\mathbf m = \mathbf m(x)$ per unit length
of the strip is equal to (see Appendix~\ref{append})
\begin{align}
  \label{Ebm}
  E_{\beta,w}(\mathbf m) = \frac12 \int_0^w \left( |m_1'|^2 + |m_2'|^2 +
  m_1^2 \right) dx + {\nu \over 4 \pi} \int_0^w \int_0^w \ln |x -
  y|^{-1} \, 
  m_\beta'(x) m_\beta'(y) \, dx \, dy,
\end{align}
where $m_\beta = \mathbf e_\beta \cdot \mathbf m$, with $\mathbf
e_\beta = (\cos \beta, \sin \beta)$, provided that
\begin{align}
  \label{mbbc}
  m_\beta(0) = m_\beta(w) = 0.
\end{align}
%Therefore we can extend $m_\beta$ by zero outside of the interval $(0,w)$ and rewrite the energy as 
%\begin{align}
%  \label{Ebm1}
%  E_{\beta,w}(\mathbf m) = \frac12 \int_0^w \left( |m_1'|^2 + |m_2'|^2 +
%  m_1^2 \right) dx + {\nu \over 4 \pi} \int_{-\infty}^\infty \int_{-\infty}^\infty \ln |x -
%  y|^{-1} \, 
%  m_\beta'(x) m_\beta'(y) \, dx \, dy.
%\end{align}
%Using Fourier transforms we now obtain
%\begin{align}
%  \label{Ebm1}
%  E_{\beta,w}(\mathbf m) = \frac12 \int_0^w \left( |m_1'|^2 + |m_2'|^2 +
%  m_1^2 \right) dx + {\nu \over 4} \int_{-\infty}^\infty  |k| | \hat m(k)|^2 \, dk,
%\end{align}
The energy in \eqref{Ebm} may also be rewritten using the operator
$\left( - {d^2 \over dx^2} \right)^{1/2}$ (acting from $H^1(\R)$ to
$L^2(\R)$ and understood via Fourier space, see Appendix~\ref{append})
\begin{align}
  \label{Ebm2}
  E_{\beta,w}(\mathbf m) = \frac12 \int_0^w \left( |m_1'|^2 + |m_2'|^2 +
  m_1^2 \right) dx + {\nu \over 4} \int_{-\infty}^\infty
  m_\beta \left( - {d^2 \over dx^2} \right)^{1/2}  m_\beta \, dx.
\end{align}
% One can use the following representation of $\left( - {d^2 \over dx^2}
%   \right)^{1/2} $ \cite[Proposition 3.3]{dinezza12}
% \begin{align}
%   \label{halflapl}
%   \left( - {d^2 \over dx^2}
%   \right)^{1/2} u(x) = {1 \over \pi} \, 
%   \dashint_{-\infty}^\infty {u(x) - u(y) \over (x - y)^2} \, dy
% \end{align}
% to represent the energy as follows
% \begin{align}
%   \label{Ebm3}
%   E_{\beta,w}(\mathbf m) = \frac12 \int_0^w \left( |m_1'|^2 + |m_2'|^2 +
%   m_1^2 \right) dx + {\nu \over 4 \pi} \int_{-\infty}^\infty m(x) 
%   \dashint_{-\infty}^\infty {m_\beta(x) - m_\beta(y) \over (x - y)^2}
%   \, dy \, dx. 
% \end{align}
% Here and everywhere below $\dashint$ denotes the principal value of
% the integral. However, in order to avoid dealing with the principal
% value one can use the following form \cite[Proposition
% 3.4]{dinezza12}:
Another useful representation of the energy in \eqref{Ebm} that
expresses the double integral in terms of $m_\beta$ rather than its
derivative is (see Appendix~\ref{append})
\begin{align}
  \label{Ebm3}
  E_{\beta,w}(\mathbf m) = \frac12 \int_0^w \left( |m_1'|^2 + |m_2'|^2 +
  m_1^2 \right) dx + {\nu \over 8 \pi} \int_{-\infty}^\infty
  \int_{-\infty}^\infty {(m_\beta(x) - m_\beta(y))^2 \over (x - y)^2}
  \, dx \, dy. 
\end{align}
Lastly, we express the energy in \eqref{Ebm3} in terms of the angle
$\theta$ between $\mathbf m$ and the easy axis in the
counter-clockwise direction:
\begin{align}
  \label{mth}
  \mathbf m = (-\sin \theta, \cos \theta).
\end{align}
With a slight abuse of notation, we get that the energy associated
with $\mathbf m$ is given by
\begin{align}
  \label{Ethb2}
  E_{\beta,w}(\theta) = \frac12 \int_0^w \left( |\theta'|^2 + \sin^2
  \theta \right) dx + {\nu \over 8 \pi} \int_{-\infty}^\infty
  \int_{-\infty}^\infty  {(\sin(\theta(x) -
  \beta) - \sin(\theta(y) - \beta))^2 \over (x - y)^2} \, dx \, dy,
\end{align}
where we set $\theta(x) = \beta$, for all $x \not\in (0, w)$.  The
energy functional in \eqref{Ethb2} is the starting point of our
analysis throughout the rest of this paper. In particular, it is
straightforward to show that minimizers of \eqref{Ethb2} exist among
all $\theta - \beta \in H^1_0(0, w)$, are smooth in the interior and
satisfy the Euler-Lagrange equation
\begin{align}
  \label{Ethb2EL}
  0 = {d^2 \theta \over dx^2} - \sin\theta\cos\theta
  -\frac{\nu}{2}\cos(\theta-\beta) \left( - {d^2 \over dx^2}
  \right)^{1/2} \sin(\theta-\beta) \qquad x \in (0, w),  
\end{align}
where \cite{dinezza12}
\begin{align}
  \label{halflapl}
  \left( - {d^2 \over dx^2}
  \right)^{1/2} u(x) = {1 \over \pi} \, 
  \dashint_{-\infty}^\infty {u(x) - u(y) \over (x - y)^2} \, dy,
\end{align}
and here and everywhere below $\dashint$ denotes the principal value
of the integral. Notice that the Euler-Lagrange equation in
\eqref{Ethb2EL} coincides with the one for the classical problem of
the N\'eel wall \cite{cm:non13}.

\section{Statement of results}
\label{sec:statement-results}

We now turn to the problem of our main interest in this paper, which
is to characterize a single edge domain wall. For this purpose, we
would like to send the parameter $w$ to infinity and obtain an energy
minimizing profile $\theta(x)$ solving \eqref{Ethb2EL} for all $x > 0$
and satisfying $\theta(0) = \beta$ (also setting $\theta(x) = \beta$
for all $x < 0$ in the definition of the last term in
\eqref{Ethb2EL}). We note that for the problem on the semi-infinite
domain with $\beta \in [0, \pi/2]$ the boundary condition at $x = 0$
is equivalent to that in \eqref{mbbc} because of the reflection
symmetry, which makes the energy invariant with respect to the
transformation 
\begin{align}
  \theta \to -\theta, \qquad \beta \to -\beta.  
\end{align}
We also note that for $\beta = 0$ we clearly have $\theta = 0$ as the
unique global minimizer for the energy in \eqref{Ethb2}. Therefore, in
the following we always assume that $\beta > 0$.

As can be seen from standard phase plane analysis, in the absence of
the nonlocal term due to stray field, i.e., when $\nu = 0$, the edge
domain wall solution is explicitly
\begin{align}
  \label{nostray}
  \theta(x) = 2 \arctan \left( e^{-x} \tan {\beta \over 2} \right)
  \qquad \text{for} \qquad x > 0,
\end{align}
noting that for $\beta = \pi/2$ there is also another solution which
is obtained from the one in \eqref{nostray} by a reflection with
respect to $\theta = \pi/2$. Furthermore, for all
$\beta \in (0, \pi/2)$ and $\nu = 0$ this is the unique solution of
\eqref{Ethb2EL} satisfying $\theta(0) = \beta$ and approaching a
constant as $x \to +\infty$. The profile $\theta(x)$ is decreasing
monotonically from $\theta = \beta$ at $x = 0$ to $\theta = 0$ at
$x = +\infty$ and decays exponentially at infinity. It also minimizes
the energy in \eqref{Ethb2} with $\nu = 0$ and $w = \infty$ among all
$\theta - \beta \in \mathring{H}^1_0(\R^+)$. By
$\mathring{H}^1_0(\R^+)$ we mean the Hilbert space obtained as the
completion of the space $C^\infty_c(\R^+)$ with respect to the
homogeneous Sobolev norm
\begin{align}
  \label{H10norm}
  \| u \|_{\mathring{H}^1_0(\R^+)}^2 := \int_0^\infty |u'|^2 dx.
\end{align}
Note that by Sobolev embedding the elements of
$\mathring{H}^1_0(\R^+)$ may be identified with continuous functions
vanishing at $x = 0$ (cf.  \cite[Section 8.3]{brezis}). The minimizing
property of $\theta(x)$ in \eqref{nostray} may be seen directly from
the Modica-Mortola type inequality for the energy with $\nu = 0$ and
$w = \infty$:
\begin{align}
  \label{Eb00}
  E_\beta^0(\theta) := \frac12 \int_0^\infty \left( |\theta'|^2 +
  \sin^2 \theta \right) dx = \int_0^\infty | (\cos \theta)'| \, dx +
  \frac12 \int_0^\infty \left( |\theta'| - |\sin \theta| \right)^2 dx
  \notag \\
  \geq \left| \int_0^\infty (\cos \theta)' \, dx \right| = | \cos
  \beta - \cos \theta_\infty| \geq 1 - \cos \beta. 
\end{align}
In writing \eqref{Eb00}, we used weak chain rule \cite[Corollary
8.11]{brezis} and the fact that $\sin \theta \in H^1(\R^+)$ whenever
$E_\beta^0(\theta) < +\infty$ and, hence, $\sin \theta(x) \to 0$ as
$x \to +\infty$, implying that
$\theta(x) \to \theta_\infty \in \pi \mathbb Z$ \cite[Corollary
8.9]{brezis}. Furthermore, by inspection the case of equality holds if
and only if $\theta$ is given by \eqref{nostray}.

A natural question is whether this type of boundary layer solution
also exists for $\nu > 0$.  We point out from the outset that if one
formally sets $w = \infty$ in \eqref{Ethb2}, one runs into a
difficulty that the nonlocal term in the energy evaluated on the
function in \eqref{nostray} is infinite. Indeed, by positivity of the
nonlocal and anisotropy terms, for any configuration with bounded
energy we would have $\sin \theta \in H^1(\R^+)$ and, therefore,
$\lim_{x \to \infty} \theta(x) = \theta_\infty \in \pi \mathbb Z$, as
before. On the other hand, if
$\theta_\infty - \beta \not\in 2 \pi \mathbb Z$, the nonlocal part of
the energy becomes infinite:
\begin{align}
  \label{log}
  \int_{-\infty}^\infty \int_{-\infty}^\infty {(\sin(\theta(x) -
  \beta) - \sin(\theta(y) - \beta))^2 \over (x - y)^2} \, dx \, dy
  \geq  \int_R^\infty \int_{-\infty}^0
  {\sin^2(\theta(y) - \beta) \over (x - y)^2} \, dx \, dy  \qquad
  \notag \\ 
  \geq \int_R^\infty {\sin^2(\theta(y) - \beta) \over y} \, dy \geq {1
  \over 2} \sin^2 \beta \int_R^\infty {dy \over y} = +\infty,
\end{align}
where we chose a sufficiently large $R > 0$, such that
$\sin^2(\theta(y) - \beta) \geq \frac12 \sin^2(\theta_\infty - \beta)
= \frac12 \sin^2 \beta > 0$
for all $y > R$. This phenomenon has to do with the divergence of the
energy of a pair of edge domain walls minimizing $E_{\beta,w}$ in
\eqref{Ethb2} as $w \to \infty$. Indeed, for $\beta \not= 0$ an edge
domain wall carries a net magnetic charge spread over a region of
width of order 1 near the edge. Therefore, the self-interaction energy
per unit length of a single edge domain wall diverges logarithmically
with $w$, as can be seen by examining the argument in
\eqref{log}. Thus, in order to concentrate on a single edge domain
wall, we need to appropriately renormalize the wall energy by
``subtracting'' the infinite self-interaction energy of a single
wall. To this end, we introduce a smooth cutoff function
$\eta_\beta:\R \to [0,\beta]$ that satisfies $\eta_\beta(x) = \beta$
when $x\leq 0$, $\eta_\beta(x) = 0$ when $x \geq 1$, and
$\eta_\beta'(x) \leq 0$ for all $x \in \R$, and formally subtract its
contribution from the integrand in the last term in
\eqref{Ethb2}. This produces the following {\em renormalized energy}
\begin{multline}
  \label{Eb0}
  E_\beta(\theta) := \frac12 \int_0^\infty \left( |\theta'|^2 + \sin^2
    \theta
  \right) dx \\
  + {\nu \over 8 \pi} \int_{-\infty}^\infty \int_{-\infty}^\infty
  {(\sin (\theta(x) - \beta) - \sin (\theta(y) - \beta))^2 - (\sin
    (\eta_\beta(x) - \beta) - \sin (\eta_\beta(y) - \beta))^2 \over (x
    - y)^2} \, dx \, dy,
\end{multline}
which is clearly finite when $\theta = \eta_\beta$.  Notice that
$\eta_\beta(x)$ mimics the edge domain wall profile near the edge and
thus has the same leading order self-energy as the edge wall, which is
what motivates its introduction in \eqref{Eb0}.

Care is needed in defining a suitable admissible class of functions
$\theta(x)$ in order to make the last term in \eqref{Eb0} meaningful,
as the integrand there may not be in $L^1(\R^2)$. The latter is
related to the logarithmic divergence at infinity of the respective
integrals mentioned earlier. Therefore, we rewrite the energy
$E_\beta(\theta)$ in an equivalent form for smooth functions
$\theta(x)$ satisfying $\theta(x) = \beta$ for all $x < 0$ and
$\theta(x) = \pi n$ for some $n \in \mathbb Z$ and all $x > R \gg 1$:
\begin{multline}
  \label{Eb}
  E_\beta(\theta) = \int_0^\infty \left( \frac12 |\theta'|^2 + \frac12
    \sin^2 \theta + {\nu \over 4 \pi} \cdot {\sin^2 (\theta - \beta) -
      \sin^2 (\eta_\beta - \beta) \over x} \right) dx \\
  + {\nu \over 8 \pi} \int_0^\infty \int_0^\infty {(\sin (\theta(x) -
    \beta) - \sin (\theta(y) - \beta))^2 \over (x - y)^2} \, dx \,
  dy \\
  - {\nu \over 8 \pi} \int_0^\infty \int_0^\infty {(\sin
    (\eta_\beta(x) - \beta) - \sin (\eta_\beta(y) - \beta))^2 \over (x
    - y)^2} \, dx \, dy,
\end{multline}
as can be verified by a direct computation. We observe that this
energy is well-defined, possibly taking value $+\infty$, on the
admissible class
\begin{align}
  \label{A}
  \mathcal{A} := \left\{ \theta \in C \big( \overline{\R^+} \big) :
  \theta - \beta \in \mathring{H}^1_0(\R^+) \right\}.
  % \mathcal{A}= \left\{ \theta \in \mathring{H}^1 \big(\R^+ \big) \cap
  % C\big(\overline{\R^+} \big) : \theta(0)=\beta, \  \sin\theta \in
  % L^2(\R^+) \right\}.  
\end{align}
Indeed, the last term in \eqref{Eb} is independent of $\theta$ and
finite (see Lemma \ref{l:Jetab}). Therefore, the main difficulty with
the definition of $E_\beta$ comes from the last term in the first line
of \eqref{Eb}. Nevertheless, as we show in Lemma \ref{l:sinx}, the
integrand in the first line of \eqref{Eb} may be bounded from below by
an integrable function that does not depend on $\theta$. This makes
the definition of $E_\beta$ in \eqref{Eb} meaningful.

We are now in the position to state our existence result for the edge
domain walls, viewed as minimizers of the one-dimensional energy
$E_\beta$ in \eqref{Eb} over the admissible class $\mathcal A$ in
\eqref{A}.

\begin{theorem}
  \label{t:exist}
  For each $\beta \in (0, \pi/2]$ and each $\nu > 0$, there exists
  $\theta \in \mathcal A$ such that
  $E_\beta(\theta) = \inf_{\tilde \theta \in \mathcal A} E(\tilde
  \theta)$.
  Furthermore, we have $\theta \in L^\infty(\R^+)$ and
  $\lim_{x \to \infty} \theta(x) = \theta_\infty$ for some
  $\theta_\infty \in \mathbb \pi \mathbb Z$.
\end{theorem}

\noindent We remark that the minimizers obtained in Theorem
\ref{t:exist} do not depend on the specific choice of
$\eta_\beta$. Indeed, denoting by $E_\beta(\theta, \eta_\beta)$ the
value of the energy for a given $\theta$ and $\eta_\beta$, we have for
any $\theta$ and $ \eta_\beta^{(1,2)}$ such that
$E(\theta, \eta_\beta^{(1,2)}) < +\infty$:
\begin{multline}
  E_\beta(\theta, \eta_\beta^{(2)}) - E_\beta(\theta,
  \eta_\beta^{(1)}) = {\nu \over 4 \pi} \int_0^\infty {\sin^2
    (\eta_\beta^{(2)} - \beta) - \sin^2
    (\eta_\beta^{(1)} - \beta) \over x} \, dx \\
  + {\nu \over 8 \pi} \int_0^\infty \int_0^\infty {(\sin
    (\eta_\beta^{(1)}(x) - \beta) - \sin (\eta_\beta^{(1)} (y) -
    \beta))^2 - (\sin (\eta_\beta^{(2)}(x) - \beta) - \sin
    (\eta_\beta^{(2)}(y) - \beta))^2 \over (x - y)^2} \, dx \, dy,
\end{multline}
which is a constant independent of $\theta$. In particular, minimizers
of $E_\beta(\cdot, \eta_\beta^{(1)})$ over $\mathcal A$ coincide with
those of $E_\beta(\cdot, \eta_\beta^{(2)})$.

The result in Theorem \ref{t:exist} should be contrasted with that for
the case $\nu = 0$ discussed at the beginning of this section. For the
latter, for all $\beta \in (0, \pi/2)$ we have existence of a unique
minimizer in $\mathcal A$, which is monotone decreasing and converging
to zero exponentially at infinity. In the case $\nu > 0$, on the other
hand, our result does not exclude a possibility of winding near the
film edge, expressed in the fact that one may have $\theta \to \pi n$
for some $n \not= 0$ as $x \to +\infty$. Similarly, neither uniqueness
nor monotonicity of the energy minimizing profile are guaranteed a
priori, and the decay at infinity is expected to follow a power law
(cf. \cite{cm:non13} and section \ref{sec:num} below).

We now turn to the questions of further regularity and the
Euler-Lagrange equation satisfied by the minimizers obtained in
Theorem \ref{t:exist}. Formally, the Euler-Lagrange equation
associated with \eqref{Eb} coincides with \eqref{Ethb2EL} for all
$x \in \R^+$ (again, extending $\theta$ to $\theta(x) = \beta$ for
$x < 0$). However, care needs to be exercised once again, since the
function $\sin (\theta - \beta)$ no longer belongs to $L^2(\R)$, so
the standard approach to the definition of
$\left( -{d^2 \over dx^2} \right)^{1/2}$ via Fourier space no longer
applies directly. Nevertheless, we show that \eqref{Ethb2EL} still
holds for the minimizers, provided that one uses the integral
representation in \eqref{halflapl} as the definition for
$\left( -{d^2 \over dx^2} \right)^{1/2}$. The latter is meaningful
whenever $\theta$ is smooth, and we have explicitly 
\begin{multline}
  \label{EL}
  \theta''(x) = \sin \theta(x) \cos \theta(x) + {\nu \over 2 \pi}
  \cdot {\sin (\theta(x) - \beta) \cos
    (\theta(x) - \beta) \over x} \\
  + {\nu \over 2 \pi} \cos (\theta(x) - \beta) \left(
    \dashint_0^\infty {\sin (\theta(x) - \beta) - \sin ( \theta(y) -
      \beta) \over (x - y)^2} \, dy \right) \qquad \forall x > 0.
\end{multline}
This picture is made precise in the following theorem.  

\begin{theorem}
  \label{t:regular}
  For each $\beta \in (0, \pi/2]$ and each $\nu > 0$, let $\theta$ be
  a minimizer from Theorem \ref{t:exist}. Then
  $\theta \in C^2(\R^+) \cap C^1(\overline{\R^+}) \cap
  W^{1,\infty}(\R^+)$
  and \eqref{EL} holds. In addition, we have
  $|\theta'(0)| = \sin \beta$ and
  $\lim_{x \to 0^+} |\theta''(x)| = \infty$.
\end{theorem}

We remark that the last statement in Theorem \ref{t:regular} prevents
the minimizer in Theorem \ref{t:exist} to be smooth up to $x = 0$, in
contrast with the case $\nu = 0$ (see \eqref{nostray}). In turn,
because of the presence of the nonlocal term in \eqref{EL} further
regularity of the minimizer for $x > 0$ cannot be established by a
standard bootstrap argument. A further study of higher regularity of
the domain wall profiles in the film interior would require a finer
simultaneous treatment of the exchange and stray field terms and goes
beyond the scope of the present paper.

We end with a consideration of two parameter regimes in which further
information can be obtained about the detailed structure of the energy
minimizing profiles. In both these regimes the nonlocal term in the
equation may be viewed as a perturbation. The first is the regime when
$\beta \in (0, \pi / 2)$ is arbitrary, but $\nu$ is sufficiently small
depending on $\beta$. Then we have the following result about the
behavior of minimizers.

\begin{theorem}
\label{t:gammanu}
Let $\beta \in (0, \pi/2)$ and let $\theta_\nu$ be a minimizer of
\eqref{Eb} for a given $\nu > 0$. Then, as $\nu \to 0$ the minimizers
$\theta_\nu$ converge uniformly on $[0, +\infty)$ to the minimizer
$\theta_0$ of \eqref{Eb00}, defined in \eqref{nostray}.
\end{theorem}

\noindent We note that we need to avoid the value of $\beta = \pi/2$
in the statement of Theorem \ref{t:gammanu}, because even in the case
$\nu = 0$ there are two possible minimizers: one is given by
\eqref{nostray} and the other by its reflection with respect to
$\theta = \pi / 2$. On the other hand, it is easy to see by an
inspection of the proof of Theorem \ref{t:gammanu} that convergence in
its statement is uniform in $\beta$ for all
$0 < \beta \leq \beta_0 < \pi/2$. Note also that the statement of
Theorem \ref{t:gammanu} implies that $\theta_\nu(x) \to 0$ as
$x \to +\infty$ for all $\nu > 0$ sufficiently small depending on
$\beta$. In other words, the domain wall profiles cannot exhibit
winding in this parameter range.

The second regime is for fixed values of $\nu > 0$ and sufficiently
small $\beta > 0$. Here we have the following result.

\begin{theorem}
  \label{t:gammabeta}
  Let $\nu >0$, let $0<\beta \leq \beta_0$ for some
  $\beta_0(\nu) > 0$, and let $\theta_\beta$ be a minimizer of
  \eqref{Eb}. Then $\theta_\beta$ is unique, and $\theta_\beta \to 0$
  uniformly on $[0, +\infty)$ as $\beta \to 0$.
\end{theorem}

\noindent Again, the statement of Theorem \ref{t:gammabeta} implies
that $\theta_\beta(x) \to 0$ as $x \to +\infty$ for all $\beta > 0$
sufficiently small depending on $\nu$.

\section{Proof of Theorem \ref{t:exist}}
\label{sec:proof-theor-reft}

We begin by defining
\begin{align}
  \label{Jbeta}
  J_\beta(\theta) := \int_{0}^{\infty}\int_{0}^{\infty}
  \frac{\brackets{\sin(\theta(x)-\beta) -
  \sin(\theta(y)-\beta)}^2}{(x-y)^2} \, dx \, dy,
\end{align}
and making the following basic observation.

\begin{lemma}
  \label{l:Jetab}
  We have
  \begin{align}
    J_\beta(\eta_\beta) < +\infty.
  \end{align}
\end{lemma}

\begin{proof}
  Since $\eta_\beta(x) = 0$ for all $x \geq 1$, we may write
  \begin{align}
    J_\beta(\eta_\beta) 
    & = \int_0^1 \int_0^1 \frac{\brackets{\sin(\eta_\beta (x)-\beta) -
      \sin(\eta_\beta (y)-\beta)}^2}{(x-y)^2} \, dx \, dy \notag \\
    & \qquad + 2 \int_0^1
      \int_1^\infty \frac{(\sin (\eta_\beta (y)-\beta) - \sin \beta)^2
      }{(x-y)^2} \,
      dx \, dy \notag \\
    & = \int_0^1 \int_0^1 \frac{\brackets{\sin(\eta_\beta (x)-\beta) -
      \sin(\eta_\beta (y)-\beta)}^2}{(x-y)^2} \, dx \, dy \notag \\
    & \qquad + 2 \int_0^1 \frac{(
      \sin(\eta_\beta (y)-\beta) - \sin \beta)^2 }{1-y} \,
      dy.     \label{Jbetaest} 
  \end{align}
  By smoothness of $\eta_\beta$ we have
  \begin{align}
    |\sin(\eta_\beta (x)-\beta) - \sin(\eta_\beta (y)-\beta)| \leq C
    |x - y|,  
  \end{align}
  for some $C > 0$ depending only on $\eta_\beta$. Therefore, the
  integrands in the right-hand side of \eqref{Jbetaest} are essentially
  bounded, yielding the result.
\end{proof}

With the result of Lemma \ref{l:Jetab} in hand, we can write the
energy in \eqref{Eb} as
\begin{align}
  E_\beta(\theta) = F_\beta(\theta) + {\nu
  \over 8 \pi}  J_\beta(\theta) -  {\nu \over 8 \pi}
  J_\beta(\eta_\beta),
\end{align}
where
\begin{align}
  \label{Fb}
  F_\beta(\theta) := \int_0^\infty \left( \frac12 |\theta'|^2 + \frac12
  \sin^2 \theta + {\nu \over 4 \pi} \cdot {\sin^2 (\theta - \beta) -
  \sin^2 (\eta_\beta - \beta) \over x} \right) dx. 
\end{align}
While $J_\beta(\theta) \geq 0$ by definition, we also have the
following result concerning $F_\beta(\theta)$.

\begin{lemma}
  \label{l:sinx}
  For every $\theta \in \mathcal A$ we have
  \begin{align}
    \label{intsinx}
    \frac12 \sin^2 \theta(x) + {\nu \over 4 \pi} \cdot {\sin^2
    (\theta(x) - \beta) - \sin^2 (\eta_\beta(x) - \beta) \over x}
    \geq - {C \over 1 + x^2} \qquad \forall x > 0,
  \end{align}
  for some $C > 0$ depending only on $\beta$, $\nu$ and
  $\eta_\beta$. Furthermore, $F_\beta(\theta)$ is bounded below
  independently of $\theta$.
\end{lemma}

\begin{proof}
  Since by the definition of $\eta_\beta$ we have
  $|\eta_\beta(x) - \beta| \leq C x$ for some $C > 0$ and all
  $x \in (0,1)$, the left-hand side of \eqref{intsinx} is bounded
  below on this interval. At the same time, by trigonometric
  identities and Young's inequality we have for all $x \geq 1$ and any
  $\eps > 0$
\begin{align}
  {\sin^2 (\theta - \beta) -
  \sin^2 (\eta_\beta - \beta) \over x} = {
  \sin(\theta - 2 \beta) \sin \theta \over x} \geq -\eps \sin^2 \theta
  - {1 \over 4 \eps x^2}.
\end{align}
Hence, choosing $\eps$ sufficiently small depending only on $\nu$, we
can control the left-hand side of \eqref{intsinx} from below by
$-C / x^2$ for all $x \in (1, \infty)$, where the constant $C > 0$
depends only on $\nu$. Combining the estimates on the two intervals
then yields \eqref{intsinx}. Finally, the last statement follows from
\eqref{intsinx} and the fact that the integrand in \eqref{Fb} is
measurable on $\R^+$.
\end{proof}

\begin{proof}[Proof of Theorem \ref{t:exist}]
  Since
  $\inf_{\tilde\theta \in \mathcal A} E_\beta(\tilde\theta) \leq
  E_\beta(\eta_\beta) < +\infty$,
  and since $E_\beta(\theta) \geq -C$ for some $C > 0$ and all
  $\theta \in \mathcal A$ by Lemmas \ref{l:Jetab} and 
  \ref{l:sinx}, there exists a sequence of $\theta_n \in \mathcal A$
  such that
  \begin{align}
    -\infty < \inf_{\tilde\theta \in
    \mathcal A} E_\beta(\tilde\theta)  = \liminf_{n \to \infty}
    E_\beta(\theta_n) < +\infty.     
  \end{align}
  Furthermore, by Lemma \ref{l:sinx} and positivity of $J_\beta$ we
  have $\| \theta_n' \|_{L^2(\R^+)} \leq C$ for some $C > 0$.
  Therefore, upon extraction of subsequences we have
  $\theta_n' \rightharpoonup \theta'$ in $L^2(\R^+)$ for some
  $\theta \in H^1_{loc}(\R^+)$ and $\theta_n \to \theta$ in
  $C([0, R])$ for any $R > 0$ fixed by Sobolev embedding \cite[Theorem
  8.8 and Proposition 8.13]{brezis}. By a diagonal argument we then
  conclude that upon further extraction of a subsequence we have
  $\theta_n(x) \to \theta(x)$ for every $x > 0$.

  Now, by the lower semicontinuity of the norm we have
  $\liminf_{n \to \infty} \| \theta_n' \|_{L^2(\R^+)} \geq \| \theta'
  \|_{L^2(\R^+)}$.
  Therefore, by Lemma \ref{l:sinx} and Fatou's lemma we get
  $\liminf_{n \to \infty} F_\beta(\theta_n) \geq F_\beta(\theta)$.
  Again, by Fatou's lemma we also get
  $\liminf_{n \to \infty} J_\beta(\theta_n) \geq J_\beta(\theta)$.  In
  fact, since $(\theta_n)$ is a minimizing sequence, the inequalities
  above are equalities. This implies that
  $\theta_n - \beta \to \theta - \beta$ strongly in
  $\mathring{H}^1_0(\R^+)$, and, hence, we have
  $\theta \in \mathcal A$. Thus, $\theta$ is a minimizer of $E_\beta$
  over $\mathcal A$.

  Finally, observe that by weak chain rule \cite[Corollary
  8.11]{brezis} and Lemma \ref{l:sinx} we have
  $\sin \theta \in H^1(\R^+)$.  Therefore, by \cite[Corollary
  8.9]{brezis} we also have
  $\sin \theta \in C\big(\overline{\R^+} \big)$ and
  $\sin \theta(x) \to 0$ as $x \to +\infty$. On the other hand, from
  the Modica-Mortola type inequality and Lemmas \ref{l:Jetab} and
  \ref{l:sinx} we obtain
  \begin{align}
    \label{mm}
    \int_0^\infty |\sin \theta| \, |\theta'| \, dx \leq \frac12 
    \int_0^\infty \brackets{|\theta'|^2 + \sin^2\theta} d x < +\infty,
  \end{align}
  which implies that $\theta(x) \to \theta_\infty$ for some
  $\theta_\infty \in \pi \mathbb Z$ as $x \to +\infty$. This concludes
  the proof.
\end{proof}

\begin{remark}
  We have shown existence of a minimizer for the energy in \eqref{Eb}
  since a priori the energy in \eqref{Eb0} did not make sense for all
  $\theta \in \mathcal A$. However, if $\theta \in \mathcal A$
  satisfies $E_\beta(\theta) <C$ (in the sense of \eqref{Eb}) then it
  is not difficult to see that the energy in \eqref{Eb0} also makes
  sense and coincides with that in \eqref{Eb}.
\end{remark}

\section{Proof of Theorem \ref{t:regular}}
\label{sec:proof-theor-reft1}

We now proceed to the proof of Theorem \ref{t:regular}, where we have
to compute a variation of the energy in \eqref{Eb}. It is clear how to
deal with all the terms except $J_\beta(\theta)$. In order to find the
variation of $J_\beta(\theta)$, for notational convenience we define
\begin{align}
  \label{u}
  u(x) := \sin (\theta(x) - \beta),
\end{align}
for $\theta \in \mathcal A$ and notice that
\begin{align}
  \label{Jb2}
  \int_0^\infty \int_0^\infty {(u(x) - u(y))^2 \over
  (x - y)^2} \, dx \, dy =   J_\beta(\theta).
\end{align}
We are taking a variation with respect to $\theta$ as follows:
$\theta(x) \mapsto \theta(x) + \eps \phi(x)$, with
$\phi \in C^\infty_c(\R^+)$ and $\eps \in \R$.  It is convenient to
introduce the corresponding variation in $u$, defined as
$u(x) \mapsto u(x) + \eps \psi_\eps(x)$, where
\begin{align}
  \label{psi}
  \psi_\eps(x) := 
  \begin{cases}
    {\sin (\theta(x) + \eps \phi(x) - \beta) - \sin (\theta(x) -
      \beta) \over \eps} & \eps \not= 0, \\
    \cos (\theta(x) - \beta) \phi(x) & \eps = 0.
  \end{cases}
\end{align}
Note that for every 
$x \in \R^+$ we have
\begin{align}
  \label{psi0}
  \psi_\eps(x) \to \psi_0(x) = : \psi(x) \
  \text{as} \ \eps 
  \to 0. 
\end{align}

Before computing the variation of $J_\beta(\theta)$ we will need two
technical lemmas concerning the properties of $\psi_\eps$.

\begin{lemma}
  \label{l:trig}
  Let $\eps \in \R$, $\theta \in \mathcal A$,
  $\phi \in C^\infty_c(\R^+)$, and let $\psi_\eps$ be defined in
  \eqref{psi}. Then $\psi_\eps \in H^1(\R^+)$, and for almost every
  $x \in \R^+$ we have
  \begin{align}
    \label{phitr}
    |\psi_\eps(x)| \leq |\phi(x)| \qquad \text{and} \qquad
    |\psi_\eps'(x)| \leq |\phi'(x)| + |\theta'(x)| \, |\phi(x)|.
  \end{align}
\end{lemma}

\begin{proof}
  By mean value theorem, we have
  \begin{align}
    \psi_\eps(x) = \cos(\theta(x) + \eps \lambda_\eps(x) \phi(x) -
    \beta) \phi(x),
  \end{align}
  for some $\lambda_\eps(x) \in (0,1)$, which yields the first
  inequality in \eqref{phitr}.  Next, applying the weak chain rule
  \cite[Corollary 8.11]{brezis}, we obtain
  $\psi_\eps\in H^1_{loc}(\R^+)$, and for almost every $x \in \R^+$ we
  have
\begin{align}
  \psi_\eps'(x) = \cos (\theta(x) + \eps \phi(x) - \beta)
  \phi'(x) +  {\cos(\theta(x) + \eps \phi(x) - \beta) -
  \cos(\theta(x) - \beta) \over \eps} \, \theta'(x).
\end{align}
In particular, $\psi_\eps'$ has compact support and, hence,
$\psi_\eps \in H^1(\R^+)$. Therefore, again by mean value theorem and
triangle inequality we have for some $\lambda_\eps(x) \in (0, 1)$
\begin{align}
   |\psi_\eps'(x)| \leq |\cos (\theta(x) + \eps \phi(x) - \beta)| \,
  |\phi'(x)| + |\sin (\theta(x) + \eps \lambda_\eps(x) \phi(x) -
  \beta) | \, |\phi(x)| \, |\theta'(x)|,
\end{align}
yielding the result.
\end{proof}

\begin{lemma}
  \label{l:h12}
  Let $\eps \in \R$, $\theta \in \mathcal A$,
  $\phi \in C^\infty_c(\R^+)$, and let $\psi_\eps$ be defined in
  \eqref{psi}. Then there exists $C > 0$ independent of $\eps$ such
  that for every $\delta > 0$ there holds
  \begin{align}
    \label{psiepsh12}
    \iint_{\{ |x - y| \leq \delta \} }  {(\psi_\eps(x) - \psi_\eps(y))^2 \over
    (x - y)^2} \, dx \, dy \leq C \delta \quad \text{and} \quad
    \iint_{ \{ |x - y|  \geq \delta \} }  {(\psi_\eps(x) - \psi_\eps(y))^2 \over 
    (x - y)^2} \, dx \, dy \leq C  \delta^{-1}.
  \end{align}
  In particular
  \begin{align}
    \label{psiepsh3}
    \int_0^\infty \int_0^\infty {(\psi_\eps(x) - \psi_\eps(y))^2 \over
    (x - y)^2} \, dx \, dy \leq 2 C.
  \end{align}
\end{lemma}

\begin{proof}
  First of all, since $\psi_\eps$ has compact support lying in $\R^+$,
  we can extend $\psi_\eps$ by zero to the whole real line. By Lemma
  \ref{l:trig}, $\psi_\eps$ is absolutely continuous and, hence, for
  all $x \not= y$ we have
  \begin{align}
    {\psi_\eps(x) - \psi_\eps(y)  \over x - y} = \int_0^1
    \psi_\eps'(y + t (x - y)) dt. 
  \end{align}
  Therefore
  \begin{align}
    & \iint_{ \{ |x - y| \leq \delta \} }  {(\psi_\eps(x) - \psi_\eps(y))^2 \over
      (x - y)^2} \, dx \, dy \notag \\
    & \hspace{2cm} = \iint_{ \{ |x - y| \leq \delta \} }
      \int_0^1 \int_0^1 \psi_\eps'(y + t (x - y)) \psi_\eps'(y + s (x -
      y)) \, dt \, ds \, dx \, dy.
      \label{iintpsi}
  \end{align}
  Interchanging the order of integration and applying Cauchy-Schwarz
  inequality, from \eqref{iintpsi} we obtain
  \begin{align}
    \iint_{ \{ |x - y| \leq \delta \} }  {(\psi_\eps(x) - \psi_\eps(y))^2 \over
    (x - y)^2} \, dx \, dy \leq \int_0^1 \iint_{ \{ |x - y|
    \leq \delta \} } |\psi_\eps'(y + s (x - y))|^2 \, dx \, dy \, ds. 
  \end{align}
  Finally, using the new variable $z = x - y$ in place of $x$ yields
  \begin{align}
    \iint_{ \{ |x - y| \leq \delta \} }  {(\psi_\eps(x) - \psi_\eps(y))^2 \over
    (x - y)^2} \, dx \, dy\leq  \int_0^1 \int_{- \delta}^{+ \delta}
    \int_{-\infty}^\infty |\psi_\eps'(y + s z)|^2 \, dy  \, dz \, ds 
    = 2 \delta \int_0^\infty |\psi_\eps'(y)|^2 \, dy,
  \end{align}
  which in view of Lemma \ref{l:trig} gives the first estimate in
  \eqref{psiepsh12}.

  To obtain the second estimate in \eqref{psiepsh12}, simply note that
  \begin{multline}
    \iint_{ \{ |x - y| \geq \delta \} } {(\psi_\eps(x) -
      \psi_\eps(y))^2 \over (x - y)^2} \, dx \, dy \leq 2 \iint_{ \{
      |x - y| \geq \delta \} } {|\psi_\eps(x)|^2 + |\psi_\eps(y)|^2
      \over
      (x - y)^2} \, dx \, dy \\
    = 4 \iint_{ \{ |x - y| \geq \delta \} } {|\psi_\eps(y)|^2 \over (x
      - y)^2} \, dx \, dy = {8 \over \delta} \int_0^\infty
    |\psi_\eps(y)|^2 dy,
  \end{multline}
  yielding the claim, once again, by Lemma \ref{l:trig}.

  Lastly, \eqref{psiepsh3} is an immediate corollary to
  \eqref{psiepsh12} with $\delta = 1$.
\end{proof}

We now establish G\^ateaux differentiability of $J_\beta(\theta)$ with
respect to compactly supported smooth perturbations of $\theta$.

\begin{lemma}
  \label{l:gateaux}
  Let $\theta \in \mathcal A$ be such that $J_\beta(\theta) < \infty$,
  let $\phi \in C^\infty_c(\R^+)$, and let $u$ and $\psi$ be defined
  in \eqref{u} and \eqref{psi0}, respectively. Then
  \begin{align}
    \lim_{\eps \to 0} {J_\beta(\theta + \eps \phi) - J_\beta(\theta)
    \over \eps} = 2 \int_0^\infty \int_0^\infty {(u(x) - u(y))
    (\psi(x) - \psi(y)) \over (x - y)^2} \, dx \, dy.
  \end{align}
\end{lemma}

\begin{proof}
  Observe that, using $\psi_\eps$ defined in \eqref{psi}, we can write
  \begin{align}
    {J_\beta(\theta + \eps \phi) - J_\beta(\theta)
    \over \eps} = 2 \int_0^\infty \int_0^\infty {(u(x) - u(y))
    (\psi_\eps(x) - \psi_\eps(y)) \over (x - y)^2} \, dx \, dy \notag 
    \\
    + \eps \int_0^\infty \int_0^\infty {  (\psi_\eps(x) - \psi_\eps(y))^2
    \over (x - y)^2} \, dx \, dy.  
  \end{align}
  Next, for $\delta \in (0, 1)$ we split the integrals above into
  those over $\{ |x - y| \leq \delta\}$ and those over
  $\{ |x - y| > \delta\}$.  Since $J_\beta(\theta) < \infty$, by
  Cauchy-Schwarz inequality and Lemma \ref{l:h12} the former are
  bounded by $C \sqrt{\delta}$ with $C \geq 0$ independent of
  $\eps \in (-1,1) \backslash\{0\}$. To compute the latter, we use
  Lebesgue dominated convergence theorem to pass to the limit as
  $\eps \to 0$ with $\delta$ fixed. Observe that since
  \begin{align}
    \label{sqdom}
    2 \iint_{\{ |x - y| > \delta \} }  {|u(x) - u(y)| \, 
    |\psi_\eps(x) - \psi_\eps(y)| \over (x - y)^2} \, dx \, dy \leq
    \iint_{\{ |x - y| > \delta \} }  {(u(x) - u(y))^2 \over (x - y)^2} \,
    dx \, dy \notag \\
    + \iint_{\{ |x - y| > \delta \} }  {
    (\psi_\eps(x) - \psi_\eps(y))^2 \over (x - y)^2} \, dx \, dy,
  \end{align}
  and since by our assumption and \eqref{Jb2} the first integral is
  bounded, it is sufficient to dominate the integrand in the second
  term of the right-hand-side of \eqref{sqdom} by an integrable
  function independent of $\eps$.

  Using Lemma \ref{l:trig}, we can write for all $|x - y| > \delta$:
  \begin{align}
    { (\psi_\eps(x) - \psi_\eps(y))^2 \over (x - y)^2} 
    \leq  {  (|\psi_\eps(x)| + |\psi_\eps(y)|)^2 \over (x - y)^2} \leq
    2 \,  {|\phi(x)|^2 + |\phi(y)|^2 \over (x - y)^2} \chi_{\{ |x - y|
    > \delta\} }(x, y) =: G_\delta(x, y),
  \end{align}
  where $\chi_{\{ |x - y| > \delta\} }$ is the characteristic function
  of the set $\{ |x - y| > \delta\}$. Since $\phi$ is bounded and has
  compact support, we have $G_\delta \in L^1(\R^+ \times \R^+)$.
  Therefore, since $\psi_\eps(x) \to \psi(x)$ as $\eps \to 0$ for all
  $x \in \R^+$, by Lebesgue dominated convergence theorem we have
  \begin{align}
    \lim_{\eps \to 0}  
    & \iint_{\{ |x - y| > \delta \} } {(u(x) - u(y)) 
      (\psi_\eps(x) - \psi_\eps(y)) \over (x - y)^2} \, dx \, dy =
      \iint_{\{ |x - y| > \delta \} }  {(u(x) - u(y)) 
      (\psi(x) - \psi(y)) \over (x - y)^2} \, dx \, dy \\
    \lim_{\eps \to 0} 
    & \iint_{\{ |x - y| > \delta \} }  {
      (\psi_\eps(x) - \psi_\eps(y))^2 \over (x - y)^2} \, dx \, dy =
      \iint_{\{ |x - y| > \delta \} }  { (\psi(x) - \psi(y))^2 \over (x
      - y)^2} \, dx \, dy < \infty.   
  \end{align}
  Lastly, combining this result with the estimates of the integrals
  over $\{ |x - y| \leq \delta\}$ and sending $\delta \to 0$ completes
  the proof, once again, by Lebesgue dominated convergence theorem,
  Lemma \ref{l:h12} and Cauchy-Schwarz inequality.
\end{proof}

With the differentiability of $J_\beta$ established in Lemma
\ref{l:gateaux}, differentiability of $E_\beta$ then follows by a
standard argument. Thus, we arrive at the following result that yields
the Euler-Lagrange equation for a minimizer of $E_\beta$ over
$\mathcal A$ in weak form.

\begin{proposition}
  \label{p:ELw}
  Let $\theta$ be a minimizer of $E_\beta$ over $\mathcal A$. Then
  \begin{align}
    \label{ELw}
    {\nu \over 4 \pi} \int_0^\infty \int_0^\infty {( \sin (\theta(x)
    - \beta) - \sin(\theta(y) - \beta)) (\cos (\theta(x) - \beta)
    \phi(x) - \cos (\theta(y) - \beta) \phi(y)) \over (x - y)^2} \, dx
    \, dy\notag \\ 
    + \int_0^\infty \left( \theta' \phi' + \phi \sin \theta \cos \theta 
    \right) dx + {\nu \over 2 \pi} \int_0^\infty {\phi \sin (\theta -
    \beta) \cos (\theta - \beta) \over x} \, dx = 0 \qquad \forall
    \phi \in C^\infty_c\big(\R^+ \big).   
  \end{align}
\end{proposition}

Our next goal is to find an alternative representation of the nonlocal
term in \eqref{ELw} that would allow us to proceed with establishing
higher regularity of the minimizers of $E_\beta$, ultimately obtaining
the classical form of the Euler-Lagrange equation in \eqref{EL}.

\begin{lemma}
  \label{l:h2}
  Let $\theta \in \mathcal A$ be such that $J_\beta(\theta) <
  \infty$. Then for every $\phi \in C^\infty_c\big(\R^+ \big)$ we have 
  \begin{align}
    \label{ELs}
    \int_0^\infty \int_0^\infty {( \sin (\theta(x)
    - \beta) - \sin(\theta(y) - \beta)) (\cos (\theta(x) - \beta)
    \phi(x) - \cos (\theta(y) - \beta) \phi(y)) \over (x - y)^2} \, dx
    \, dy\notag \\ 
    + 2 \int_0^\infty {\sin(\theta(x) - \beta) \cos(\theta(x) - \beta) 
    \over x} \, \phi(x) \, dx  \notag \\
    = 2 \int_0^\infty \left( \dashint_0^\infty {\cos (\theta(y) -
    \beta) \theta'(y) \over x - y} \,
    dy \right) \cos (\theta(x) - \beta) \phi(x) \, dx.
  \end{align}
\end{lemma}

\begin{proof}
  We begin by defining $u$ and $\psi$ as in \eqref{u} and \eqref{psi},
  respectively, and extending them by zero to the whole of $\R$. We
  also similarly extend $\phi$. To simplify the notations, we still
  denote those extensions as $u$, $\psi$ and $\phi$, respectively. 

  Next, we define
  \begin{align}
    \label{I}
    I := \int_{-\infty}^\infty \int_{-\infty}^\infty {( u(x) - 
    u(y)) (\psi(x) - \psi(y))
    \over (x - y)^2} \, dx 
    \, dy,
  \end{align}
  and for $\delta > 0$ we write $I = I_1^\delta + I_2^\delta$, where
  \begin{align}
    I_1^\delta := \iint_ {\{ |x - y| > \delta \} }{ (  u(x) - 
    u(y)) (\psi(x)
     - \psi(y))
    \over (x - y)^2} \, dx 
    \, dy, \label{I1} \\
    I_2^\delta := \iint_ {\{ |x - y| \leq \delta \} } { (   u(x) - 
    u(y)) (\psi(x)
    - \psi(y) )
    \over (x - y)^2} \, dx 
    \, dy. \label{I2}
  \end{align}
  Note that by our assumptions \eqref{I} and, hence, \eqref{I1} and
  \eqref{I2}, define absolutely convergent integrals. Indeed, by Lemma
  \ref{l:trig} and Cauchy-Schwarz inequality we have
  \begin{align}
    |I| 
    & \leq J_\beta^{1/2}(\theta) \left( \int_{-\infty}^\infty
      \int_{-\infty}^\infty {(\psi(x) - 
      \psi(y))^2 \over (x - y)^2} \, dx  \, 
      dy \right)^{1/2} \notag \\ 
    & = J_\beta^{1/2}(\theta) \left( \int_0^\infty
      \int_0^\infty {(\psi(x) - 
      \psi(y))^2 \over (x - y)^2} \, dx  \, 
      dy + 2 \int_0^\infty {|\psi(x)|^2 \over x} \, dx \right)^{1/2},
  \end{align}
  which is finite by Lemma \ref{l:h12}.

  Let now $w_n \in C^\infty_c(\R^+)$ be such that
  $w_n \to \theta - \beta$ in $\mathring{H}^1_0(\R^+)$, and extend
  $w_n$ by zero for $x < 0$. We claim that if $u_n := \sin w_n$, then
  we have $I_n \to I$ as $n \to \infty$, where
  \begin{align}
    I_n := \int_{-\infty}^\infty \int_{-\infty}^\infty {(   u_n(x) - 
     u_n(y)) (\psi(x)
     - \psi(y) )
    \over (x - y)^2} \, dx 
    \, dy.
  \end{align}
  Indeed, define $I^\delta_{1,n}$ and $I^\delta_{2,n}$ as in
  \eqref{I1} and \eqref{I2} with $u$ replaced by $u_n$.  Arguing as in
  the proof of Lemma \ref{l:gateaux}, in view of the pointwise
  convergence of $u_n$ to $u$ we have $I^\delta_{1,n} \to I_1^\delta$.
  At the same time, by the argument in the proof of Lemma \ref{l:h12}
  and boundedness of $u_n'$ in $L^2(\R)$, we also have
  $|I^\delta_{2,n}| \leq C \delta$ for some $C > 0$ independent of
  $n$. Thus, the claim follows by arbitrariness of $\delta$.

  Now, since both $u_n$ and $\psi $ belong to $H^1(\R)$ by Lemma
  \ref{l:trig} and \cite[Corollary 8.10 and Corollary 8.11]{brezis},
  we may proceed by expressing $I_n$ in Fourier space \cite[Theorem
  7.12]{lieb-loss}:
  \begin{align}
    \label{InF}
    I_n = \int_{-\infty}^\infty |k| \mathcal F^*[ 
    u_n] \mathcal F[ \psi ] dk,
  \end{align}
  where ``$*$'' stands for complex-conjugate and
  \begin{align}
    \label{F}
    \mathcal F[u] := \int_{-\infty}^\infty e^{-i k x} u(x) \, dx.
  \end{align}
  Alternatively, we can write \eqref{InF} as 
  \begin{align}
    I_n = \int_{-\infty}^\infty (-i \, \text{sgn} (k))^{*}
    \mathcal F^*[ u_n'] \mathcal F[ \psi ] dk,
  \end{align}
  We now pass to the limit $n \to \infty$ by taking advantage of the
  fact that $u_n' \to u'$ in $L^2(\R)$ and, therefore, we have
  \begin{align}
    \label{hilb}
    I = \int_{-\infty}^\infty (-i \, \text{sgn} (k))^{*}
    \mathcal F^*[ u'] \mathcal F[ \psi ] dk = 2 
    \int_{-\infty}^\infty \left( \dashint_{-\infty}^\infty {
    u'(y) \over x - y} \, dy \right) \psi(x)  \,
    dx,
  \end{align}
  where we inserted the definition of Hilbert transform \cite[Section
  5.1.1]{grafakos}.

  Finally, to obtain the desired formula, we separate out the
  contribution of the negative real axis in the integral over $y$ in
  \eqref{I}. We have
  \begin{align}
    I = \int_0^\infty \int_0^\infty {(  u(x) - 
    u(y)) (\psi(x)
   - \psi(y) )
    \over (x - y)^2} \, dx 
    \, dy + 2 \int_{-\infty}^0 \int_0^\infty { u(x) \psi(x)
    \over (x - y)^2} \, dx 
    \, dy \notag \\
    = \int_0^\infty \int_0^\infty {(  u(x) - 
    u(y)) (\psi(x)
     - \psi(y) )
    \over (x - y)^2} \, dx 
    \, dy + 2 \int_0^\infty {u(x) \psi(x)  \over x} \, dx,
  \end{align}
  where we took into account that since
  $\phi \in C^\infty_c\big( \R^+ \big)$ (see \eqref{psi0} for the
  definition of $\psi$), there is no singularity in the integral near
  $x = 0$. At the same time, from \eqref{hilb} we get
  \begin{align}
    I = 2 \int_0^\infty \left( \dashint_0^\infty {
    u'(y) \over x - y} \, dy \right) \psi(x)  \,
    dx,
  \end{align}
  which concludes the proof.
\end{proof}

% \begin{proposition}
%   \label{p:EL}
%   Let $\theta \in \mathcal A$ satisfy \eqref{ELw}. Then
%   $\theta \in C^\infty\big( \R^+ \big)$ and \eqref{EL} holds.
% \end{proposition}

\begin{proof}[Proof of Theorem \ref{t:regular}]
  By Proposition \ref{p:ELw}, $\theta$ solves \eqref{ELw}. At the same
  time, by Lemma \ref{l:h2} we have
  \begin{align}
    \label{th2}
    \int_0^\infty \theta \phi'' \, dx = \int_0^\infty g \phi \, dx
    \qquad \forall \phi \in C^\infty_c(\R^+),
  \end{align}
  where 
  \begin{align}
    \label{gH}
    g(x) := \sin \theta(x) \cos \theta(x) + {\nu \over 2 \pi}
    \cos(\theta(x) - \beta) \, \dashint_0^\infty {\cos (\theta(y) -
      \beta) \theta'(y) \over x - y} \, dy.
  \end{align}
  Notice that since $\sin \theta \in L^2(\R^+)$ and
  $\theta' \in L^2(\R^+)$, by the properties of Hilbert transform
  \cite[Theorem 5.1.7 and Theorem 5.1.12]{grafakos} we have that
  $g \in L^2(\R^+)$ as well. Hence $\theta'' \in L^2(\R^+)$, which by
  Sobolev embedding \cite[Theorem 8.8]{brezis} implies that
  $\theta \in C^1(\overline{\R^+})$ and $\theta' \in
  L^\infty(\R^+)$. 

  Focusing now on the nonlocal term, let $u(x)$ be defined by
  \eqref{u}, extended, as usual, by zero to $x < 0$, and let $h(x)$
  denote the integral in \eqref{gH}, with $x \in \R$. Note that by
  chain rule we have $u \in C^1(\overline{\R^+})$, and $u'(x)$
  experiences a jump discontinuity at $x = 0$ whenever
  $u'(0^+) \not= 0$.  Also, by weak chain rule \cite[Corollary
  8.11]{brezis} we have $u'' \in L^2(\R^+)$. Passing to the Fourier
  space as in the proof of Lemma \ref{l:h2} (but treating $h$ as a
  member of $\mathcal S'(\R)$ now), for any
  $\phi \in C^\infty_c(\R^+)$, extended by zero to $x < 0$, we can
  write
  \begin{multline}
    -\int_{-\infty}^\infty h \phi' \, dx = \frac12
    \int_{-\infty}^\infty |k| \mathcal F^*[\phi] \mathcal F[u'] \, dk
    = \frac12 \int_{-\infty}^\infty (-i \, \text{sgn} (k)) \mathcal
    F^*[\phi] \mathcal F[u''] \,
    dk   \\
    + \frac{u'(0^+)}{2} \int_{-\infty}^\infty (-i \, \text{sgn} (k))
    \mathcal F^*[\phi] \, dk = \int_0^\infty \left( \,
      \dashint_0^\infty {u''(y) \over x - y} \, dy \right) \phi(x) \,
    dx + u'(0^+) \int_0^\infty {\phi(x) \over x} \, dx,
  \end{multline}
  where by $u''$ we mean the absolutely continuous part of the 
  distributional derivative of $u'$ on $\R$. Thus, we have
  \begin{align}
    h'(x) =  \dashint_0^\infty {u''(y) \over x - y} \, dy +
    {u'(0) \over x} \qquad \text{in} \quad \mathcal D'(\R^+). 
  \end{align}
  In this expression, the first term is still in $L^2(\R)$. Similarly,
  the second term is in $L^2_{loc}(\R^+)$. By \eqref{th2} and weak
  product and chain rules \cite[Corollary 8.9 and Corollary
  8.11]{brezis}, this then yields that
  $\theta''' \in L^2_{loc}(\R^+)$, implying, in particular, that
  $\theta \in C^2(\R^+)$. Furthermore, by \eqref{th2} we have
  \begin{align}
    \label{th3}
    \theta''(x) = g(x) \qquad \forall x > 0.
  \end{align}
  Finally, again by passing to Fourier space \cite[Section
  3]{dinezza12} we have
  \begin{align}
    \dashint_0^\infty {u'(y) \over x - y} \, dy = \dashint_{-\infty}^\infty
    {u(x) - u(y) \over (x - y)^2} \, dy = \dashint_0^\infty
    {u(x) - u(y) \over (x - y)^2} \, dy + {u(x) \over x},
  \end{align}
  where we took into account that $u(0) = 0$.  Substituting this
  expression to \eqref{th3} yields \eqref{EL}.

  To prove the remaining statements about the behavior of $\theta(x)$
  near $x = 0$, we multiply \eqref{th3} by $\theta'(x)$ and integrate
  over $\R^+$. Since the Hilbert transform is an anti-hermitian
  operator from $L^2(\R)$ to $L^2(\R)$ \cite[Section 5.1.1]{grafakos},
  and since $\cos (\theta(x) - \beta) \theta'(x)$, extended by zero
  for $x < 0$, belongs to $L^2(\R)$, the contribution of the last term
  in the right-hand side of \eqref{gH} vanishes. At the same time,
  since $\theta' \in H^1(\R^+)$ we have $\theta'(x) \to 0$ as
  $x \to +\infty$ \cite[Corollary 8.9]{brezis}. Since also
  $\sin \theta(x) \to 0$ as $x \to +\infty$, by \cite[Theorem
  8.2]{brezis} we have
  \begin{align}
    |\theta'(0)|^2 = \sin^2 \theta(0).
  \end{align}
  The boundary condition then implies that
  $ |\theta'(0)| = \sin \beta$. In particular, $\theta'(0) \not= 0$.
  Thus, the function $u'(x)$ defined above experiences a jump
  discontinuity at $x = 0$, leading to a logarithmic divergence of the
  integral in \eqref{gH} as $x \to 0^+$. This concludes the proof.
\end{proof}

\section{Proof of Theorem~\ref{t:gammanu}}
\label{sec:proof-theor-gammanu}

With a slight abuse of notation we denote the energy in \eqref{Eb} by
$E_\beta^\nu$. We first show that
$\theta_\nu - \beta \to \theta_0 -\beta$ in $\mathring{H}^1_0(\R^+)$
as $\nu \to 0$.  Let us assume that $E_\beta^\nu (\theta_\nu) \leq C$
and $0<\nu<1$. Using the same arguments as in Lemma~\ref{l:sinx}, we
obtain
\begin{align}
  % \label{intsinx}
  \frac14 \sin^2 \theta(x) + {\nu \over 4 \pi} \cdot {\sin^2
  (\theta(x) - \beta) - \sin^2 (\eta_\beta(x) - \beta) \over x}
  \geq - {C \over 1 + x^2} \qquad \forall x > 0,
\end{align}
where $C > 0$ depends only on $\beta$ and $\eta_\beta$. Therefore we
have
\begin{equation}
E_\beta^0 (\theta_\nu) = \frac12 \int_0^\infty \left( |\theta_\nu'|^2 +
  \sin^2 \theta_\nu \right) dx \leq C,
\end{equation}
which immediately implies (see the proof of Theorem~\ref{t:exist})
that $\theta_\nu - \beta \rightharpoonup \theta -\beta $ in
$\mathring{H}^1_0(\R^+)$ and $\theta \in \mathcal A$.

Now we prove $\Gamma$-convergence of energies with respect to the weak
convergence in $\mathring{H}^1_0(\R^+)$ (for a general introduction to
$\Gamma$-convergence, see, e.g., \cite{braides}). Let us assume that
$\nu_n \to 0$ and $\theta^n - \beta \rightharpoonup \theta -\beta$ in
$\mathring{H}^1_0(\R^+)$. Then by Sobolev embedding \cite[Theorem
8.8]{brezis}, upon extraction of a subsequence we also have
$\theta^n(x) \to \theta(x)$ for all $x > 0$. Therefore, by lower
semicontinuity of the norm, Fatou's lemma and positivity of $J_\beta$
we have
\begin{equation}
  \liminf_{n \to \infty} E_\beta^{\nu_n} (\theta^n) \geq
  E_\beta^0 (\theta). 
\end{equation}
Now, taking any $\theta - \beta \in \mathring{H}^1_0(\R^+)$ such that
$E_\beta^{\nu_n}(\theta) < +\infty$, we can construct a sequence
$\theta^n \equiv \theta$ that trivially satisfies
\begin{equation}
  \limsup_{n \to \infty} E_\beta^{\nu_n} (\theta^n) =
  E_\beta^0 (\theta), 
\end{equation}
establishing the $\Gamma$-limit sought.

Using the properties of $\Gamma$-convergence \cite{braides}, we then
have
$\lim_{\nu \to 0} E_\beta^\nu (\theta_\nu) = E_\beta^0 (\theta_0)$,
where $\theta_0$ is given by the right-hand side of \eqref{nostray},
which implies $\theta_\nu - \beta \to \theta_0 -\beta$ in
$\mathring{H}^1_0(\R^+)$ and $\sin\theta_\nu \to \sin\theta_0$ in
$H^1(\R^+) \cap C(\overline{\R^+})$. However, from the Modica-Mortola
type inequality in \eqref{mm} we also have for some $C>0$ depending
only on $\beta$ and $\eta_\beta$:
\begin{align}
  \int_0^\infty |\sin\theta_\nu||\theta'|\, dx \leq
  E_\beta^\nu(\theta_\nu) + \frac{C}{2} \nu \leq E_\beta^0 (\theta_0)
  + C\nu = 1 -\cos\beta + C\nu . 
\end{align}
This implies that $\theta_\nu(x) \in (-C \nu, \beta + C \nu)$ for all
$x > 0$ and some $C>0$ depending only on $\beta$ and $\eta_\beta$.
Hence for any fixed $\beta \in (0, \frac12 \pi)$ we can always choose
$\nu_0 > 0$ such that for all $\nu<\nu_0$ we have
$\theta_\nu(x) \in [-\frac12 \beta - \frac14 \pi, \frac12 \beta +
\frac14 \pi] \subset (-\frac12 \pi, \frac12 \pi)$
for all $x > 0$. Recall also that
$\theta_0(x) \in (0, \beta) \subset [-\frac12 \beta - \frac14 \pi,
\frac12 \beta + \frac14 \pi]$
for all $x > 0$. Thus, using mean value theorem, with some
$\tilde\theta(x)$ between $\theta_\nu(x)$ and $\theta_0(x)$, we arrive
at
\begin{align}
  |\sin \theta_\nu(x) - \sin \theta_0(x)| = |\cos \tilde\theta(x)|\, |
  \theta_\nu(x) - \theta_0(x)| \geq C |
  \theta_\nu(x) - \theta_0(x)| \qquad \forall x > 0,
\end{align}
for some $C > 0$ depending only on $\beta$. In particular, in view of
the uniform convergence of $\sin\theta_\nu$ to $\sin\theta_0$, for any
$\eps > 0$ and all $\nu$ sufficiently small we have
\begin{align}
  \sup_{x \in \R} |\theta_\nu(x) - \theta_0(x)| \leq
  C \sup_{x \in \R} 
  |\sin\theta_\nu(x) - \sin\theta_0(x)| <\eps.
\end{align}
This concludes the proof. \qed

\section{Proof of Theorem~\ref{t:gammabeta}}
\label{sec:proof-theor-gammabeta}

Let us first show that $\theta_\beta \to 0$ as $\beta \to 0$ uniformly
in $C(\overline{\R^+})$. We define
$\phi_\beta(x) = \max\{ \beta (1- x), 0\}$ for $0<\beta<{\pi \over 4}$
and all $x \in \R$. It is clear that
$E_\beta (\theta_\beta) \leq E_\beta (\phi_\beta)$ and, therefore,
after a straightforward computation,
\begin{align}
  \frac12 \int_0^\infty \left( |\theta_\beta'|^2 + 
  \sin^2 \theta_\beta \right)\, dx 
  &\leq  \frac12 \int_0^\infty \left( |\phi_\beta'|^2 + 
    \sin^2 \phi_\beta \right)\, dx +   {\nu \over 4 \pi} \int_0^\infty {\sin^2
    (\phi_\beta - \beta) - 
    \sin^2 (\theta_\beta - \beta) \over x} \, dx  \notag\\
  &+ {\nu \over 8 \pi} \int_0^\infty \int_0^\infty {(\sin (\phi_\beta(x) -
    \beta) - \sin (\phi_\beta(y) - \beta))^2 \over (x - y)^2} \, dx \,
    dy \notag \\  
  & \leq \beta^2 +   {\nu \over 4 \pi} \int_0^\infty {\sin^2
    (\phi_\beta - \beta) - 
    \sin^2 (\theta_\beta - \beta) \over x} \, dx  \notag\\
  &+ {\nu \over 8 \pi} \int_0^\infty \int_0^\infty {(\phi_\beta(x) -
    \phi_\beta(y))^2 \over (x - y)^2} \, dx \,
    dy \notag \\  
  &\leq \left(  {\nu \over 4 \pi} +1 \right) \beta^2 +  {\nu \over 4
    \pi} \int_0^\infty {\sin^2 (\phi_\beta - \beta) - 
    \sin^2 (\theta_\beta - \beta) \over x} \, dx.  \label{Itf}
\end{align}
Now, with the help of trigonometric identities and Young's inequality
we estimate the last integral in \eqref{Itf}:
\begin{align}
  & {\nu \over 4
    \pi} \int_0^\infty {\sin^2 (\phi_\beta - \beta) - 
    \sin^2 (\theta_\beta - \beta) \over x} \, dx \notag \\
  & \leq {\nu \over 4 \pi} \int_0^1 {\sin^2 (\phi_\beta - \beta) \over 
    x} \, dx + {\nu \over 4 \pi} \int_1^\infty {\sin (2\beta
    -\theta_\beta) \sin \theta_\beta  \over x} \, dx \notag \\ 
  & \leq {\nu \beta^2 \over 4 \pi} +  {\nu \over 4 \pi} \int_1^\infty
    {\sin 2\beta \cos \theta_\beta \sin \theta_\beta - \cos 2\beta
    \sin^2 \theta_\beta \over x} \, dx \notag \\ 
  & \leq  {\nu \beta^2 \over 4 \pi} + {\nu \beta \over 2 \pi}
    \int_1^\infty 
    {|\sin \theta_\beta | \over x}  \, dx \leq {\nu  (\nu +
    \pi) \beta^2 \over 4 \pi^2} + \frac{1}{4} \int_1^\infty
    \sin^2\theta_\beta  \, dx, 
\end{align}
recalling that $0 < \beta < {\pi \over 4}$ and, therefore,
$\cos 2 \beta > 0$.  Combining the above inequalities, we obtain
\begin{align}
  \label{Cb214}
  \frac14 \int_0^\infty \left( |\theta_\beta'|^2 + 
  \sin^2 \theta_\beta \right)\, dx \leq C \beta^2,
\end{align}
for some $C > 0$ depending only on $\nu$. Hence the left-hand side of
\eqref{Cb214} vanishes as $\beta \to 0$, and by \eqref{mm} this
implies $\theta_\beta \to 0$ uniformly in $C(\overline{\R^+})$ in this
limit.

Now we prove uniqueness of minimizers when $\beta$ is small
enough. From the arguments above it is clear that for all $\delta > 0$
and all $\beta$ sufficiently small we have
$\theta_\beta \in (-\delta, \delta)$. Therefore, if
$u_\beta := \sin(\theta_\beta -\beta)$, then
$\theta_\beta = \beta + \arcsin u_\beta \in C^1(\overline{\R^+})$ and
$u_\beta \in (-2 \delta, 2 \delta)$ for all sufficiently small
$\delta$. We rewrite the energy $E_\beta(\theta_\beta)$ in terms of
$u_\beta$:
\begin{align}
  E_\beta(\theta_\beta) 
  &= \frac12 \int_0^\infty \left( \frac{|u_\beta'|^2}{1-u_\beta^2} + 
    \sin^2(\beta + \arcsin u_\beta) + {\nu \over 4 \pi} \cdot
    \frac{u_\beta^2 - \sin^2(\eta_\beta -\beta)}{x} \right) dx \notag \\ 
  &+ {\nu \over 8 \pi} \int_0^\infty \int_0^\infty {(u_\beta(x) -
    u_\beta(y))^2 \over (x - y)^2} \, dy \, 
    dx \notag \\
  & - {\nu \over 8 \pi} \int_0^\infty \int_0^\infty {(\sin
    (\eta_\beta(x) - \beta) - \sin (\eta_\beta(y) - \beta))^2 \over (x 
    - y)^2} \, dy \, dx. \label{Econvex}
\end{align}
It is straightforward to show that the right-hand side of
\eqref{Econvex} is a strictly convex functional of $u_\beta$ for all
$\delta$ sufficiently small and, therefore, the minimizer of $E_\beta$
is unique when $\beta$ is small enough. \qed

\section{Numerics and discussion}
\label{sec:num}

We conclude this paper by presenting the results of some numerical
simulations that exhibit edge domain walls and discussing some of
their distinctive characteristics. We first perform a micromagnetic
simulation of the remanent magnetization configuration in a
ferromagnetic strip, using the simplified two-dimensional thin film
model from section \ref{sec:model} (for details of the numerical
algorithm, see \cite{mo:jcp06}). The result of the simulation is
presented in Figure~\ref{f:strip2d} and shows the long-time asymptotic
stationary magnetization configuration formed as the result of solving
the overdamped Landau-Lifshitz-Gilbert equation with the energy from
\eqref{Emd} in a strip of lateral dimensions $128.25 \times 32.25$ (in
the units of $L$). The initial condition was taken in the form of the
magnetization saturated to the direction slightly away from
vertical. The thin film parameter was fixed at $\nu = 20$.  The
dimensionless parameters above correspond, for example, to those of a
permalloy strip ($M_s = 8 \times 10^5$ A/m, $A = 1.3 \times 10^{-11}$
J/m and $K = 5 \times 10^2$ J/m$^3$ \cite{nistmumag}) with dimensions
$20.7 \mu$m$\times 5.2\mu$m$\times 4$nm, for which $\ell = 5.69$ nm
and $L = 161$ nm, giving an edge domain wall width of order 1 $\mu$m.

To obtain the one-dimensional domain wall profiles numerically, we
solve a parabolic version of \eqref{EL} for $\theta = \theta(x, t)$ in
$\R^+ \times \R^+$:
\begin{align}
  \label{ELt}
  \theta_t = \theta_{xx} - \sin \theta \cos \theta
  - {\nu \over 2} \cos (\theta - \beta) \left( -{d^2
  \over dx^2} \right)^{1/2} \sin (\theta - \beta) ,
\end{align}
with Dirichlet boundary condition
\begin{align}
  \label{th0t}
  \theta(0, t) = \beta,
\end{align}
and initial data
\begin{align}
  \label{th0x}
  \theta(x,0) = \frac{2 \beta}{1+e^{x/2}},
\end{align}
which is a monotonically decreasing function that asymptotes to zero
as $x \to +\infty$.  To solve the above problem, we employ a
finite-difference discretization and an optimal-grid-based method to
compute the stray field, extending $\theta$ by its boundary value to
$x < 0$. The details of the numerical method can be found, once again,
in \cite{mo:jcp06}. The domain wall profiles are then identified with
the steady states of \eqref{ELt} reached as $t \to \infty$.

\begin{figure}[t]
  \begin{center}
    \includegraphics[width=.9\textwidth]{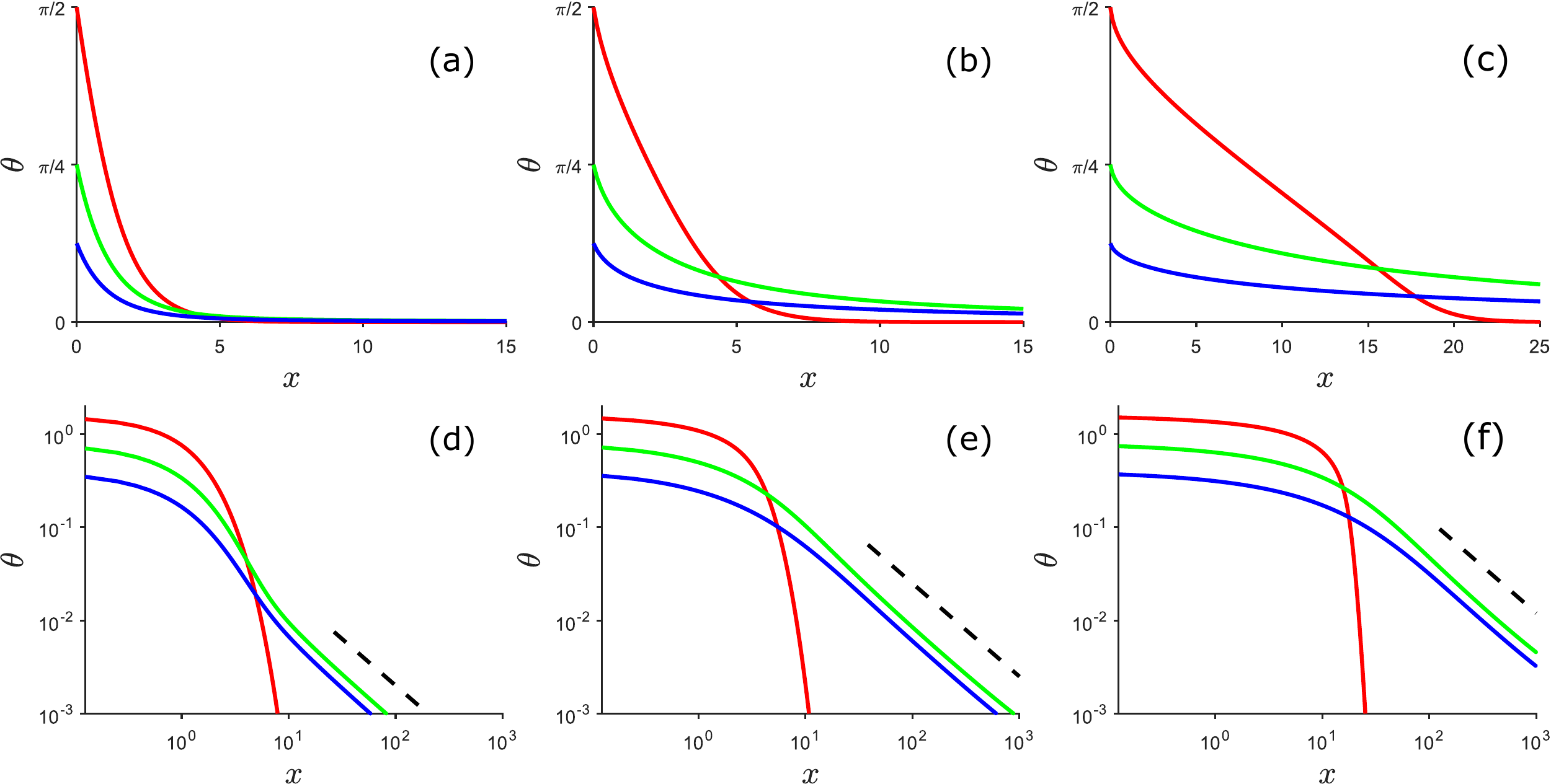}
    \caption{Computed boundary wall profiles for
      $\beta = \pi/2, \pi/4, \pi/8$. In panels (a) and (d), $\nu=1$;
      in panels (b) and (e), $\nu=10$; in panels (c) and (f),
      $\nu=50$. The upper panels (a)-(c) show the computed profiles
      for the given values of $\beta$ and $\nu$. The lower panels
      (d)-(f) show the same profiles in log-log coordinates, with the
      dashed lines indicating the $1/x$ decay.}
    \label{boundarywalls} 
  \end{center} 
\end{figure}

We collect the results of our numerical solution of the above problem
for a range of physically relevant values of $\beta$ and $\nu$ in
Figs. \ref{boundarywalls} and \ref{windingnonmono}. The upper panels
(a)-(c) of Fig. \ref{boundarywalls} show the profiles of edge domain
walls with $\beta$ equal to $\pi/2$ (red curves), $\pi/4$ (green
curves) and $\pi/8$ (blue curves). The lower panels show the same
profiles in log-log coordinates, with the dashed lines indicating
$1/x$ decay. Each pair of panels corresponds to a different value of
$\nu$: $\nu = 1$ in panels (a) and (d); $\nu = 10$ in panels (b) and
(e); $\nu = 50$ in panels (c) and (f). In all the simulations, a
discretization step $\Delta x = 0.125$ was used near the edge on a
non-uniform grid with a stretch factor $b = 20$ \cite{mo:jcp06} and
terminating at $x \simeq 6 \times 10^3$.  A 16-node optimal geometric
grid was used in the transverse direction to compute the stray field
\cite{ingerman00,mo:jcp06}.

\begin{figure}
  \centering
    \includegraphics[width=\textwidth]{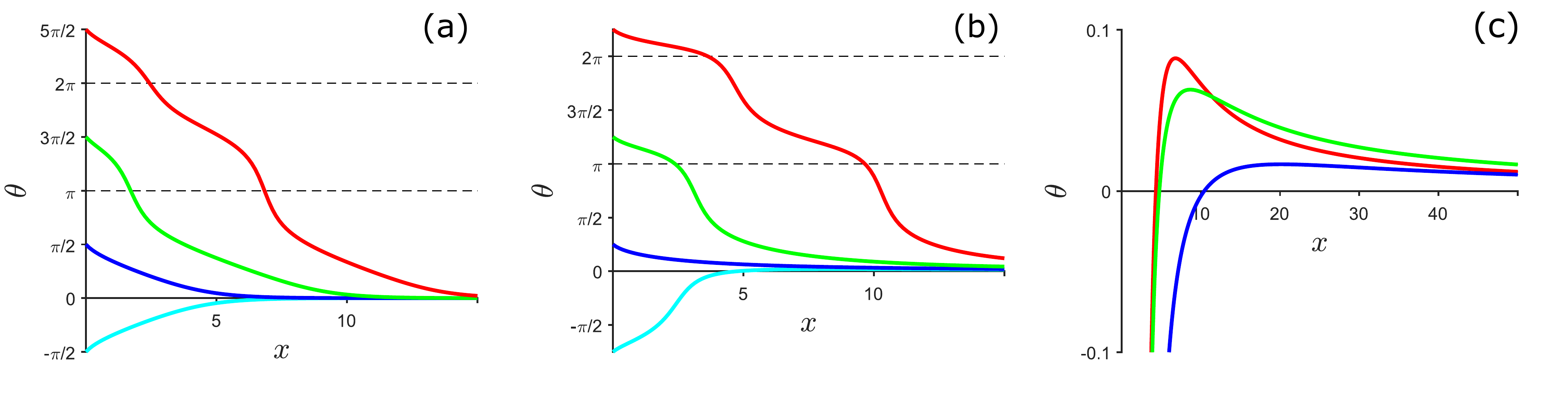} 
    \caption{Edge domain walls exhibiting winding and lack of
      monotonicity obtained by solving \eqref{ELt} for $\nu = 10$ and
      different values of $\beta$. In (a),
      $\beta = -\pi/2, \pi/2, 3 \pi/2, 5 \pi/2$. In (b),
      $\beta = -3 \pi / 4, \pi/4, 5 \pi / 4, 9 \pi/4$. In (c), the
      non-monotone decay in the tails of the solutions for
      $\beta = -5\pi/8$ (red), $\beta = -3\pi/4$ (green) and
      $\beta = -7\pi/8$ (blue) at large $x$ is emphasized. }
  \label{windingnonmono}
\end{figure}

For $0 < \beta \leq \pi/2$ the obtained domain wall profiles exhibit
monotonic decay from $\theta = \beta$ at $x = 0$ to $\theta = 0$ at
$x = +\infty$.  Thus, qualitatively the profiles are similar to those
in \eqref{nostray} corresponding to the case $\nu = 0$. This is in
agreement with the predictions of Theorem \ref{t:gammanu} and Theorem
\ref{t:gammabeta} for the cases $\nu \lesssim 1$ and
$\beta \lesssim {\pi \over 2}$, respectively. At the same time, one
can see from Figs. \ref{boundarywalls}(b) and \ref{boundarywalls}(c)
that as the value of $\nu$ is increased, the profiles develop a
multiscale structure, whereby an inner core forms near the edge on the
$O(\nu)$ length scale for $\nu \gg 1$, followed by either an
exponential ($\beta = \pi/2$) or an algebraic ($\beta \not= \pi/2$)
tail. Heuristically, this scaling may be derived by balancing the
anisotropy and the stray field terms in \eqref{EL}. A structure
similar to this was reported previously for N\'eel walls at large
values of $\nu$ \cite{garcia99,melcher03,garcia04,mo:jcp06,hubert}.
% Nevertheless, none of the profiles obtained by us exhibit winding in
% the form of $\theta_\infty = \lim_{x \to \infty} \theta(x) \not= 0$,
% a possibility which is not a priori excluded by Theorem
% \ref{t:exist}.
Notice that all the profiles are regular and exhibit a finite slope
near the edge, in agreement with Theorem \ref{t:regular}.

Focusing on the decay of the domain wall profiles, we observe that
even though the overall shape of the profile may be qualitatively
similar to that in \eqref{nostray}, they exhibit slow algebraic decay
for all $\beta \not= \pi/2$, even for small values of $\nu$, see
Figs. \ref{boundarywalls}(d)--(f). In all those cases, the profiles
exhibit a decay rate proportional to $1/x$, which can be explained by
the fact that there is a build-up of magnetic charge near the material
edge, which results in a stray field decaying like $1/x$ away from the
edge. This should be contrasted to the asymptotic $1/x^2$ decay
observed in N\'eel walls \cite{cm:non13}. At the same time, for the
special value of $\beta = \pi/2$ the decay becomes exponential, which
can also be seen directly from \eqref{EL}. Indeed, when
$\beta = \pi/2$, the $\cos (\theta(x) - \beta)$ factor multiplying the
contribution of the stray field vanishes as $x \to +\infty$, making
the anisotropy term dominate at large values of $x$ and, therefore,
resulting in exponential decay.

We note that the domain wall profiles obtained by us numerically in
Fig. \ref{boundarywalls} are not guaranteed to be those corresponding
to the global energy minimizers in Theorem \ref{t:exist}. Instead,
they may correspond only to local energy minimizers. In fact, neither
monotonicity, nor uniqueness of the energy minimizing edge domain wall
profiles are known a priori, in contrast to the N\'eel walls in the
bulk of the material \cite{cm:non13,my:prsla16}. In order to assess
whether other types of local minimizers may exist in the problem under
consideration, we performed further numerical studies of solutions of
\eqref{ELt}--\eqref{th0x} by considering values of $\beta$ outside the
interval $(0, \pi/2]$. The obtained stationary solutions for
$\nu = 10$ are shown in Fig. \ref{windingnonmono}. All these solutions
decay to zero as $x \to \infty$, indicating a nontrivial winding
(i.e., a variation of $\theta$ by more than $\pi/2$) in each one for
$\beta \not\in [0, \pi/2]$. We emphasize that these domain wall
profiles are stabilized by nonlocal effects, since in the absence of
stray field, i.e., when $\nu = 0$, such solutions do not exist. At the
same time, the solutions with winding appear to have higher energy
than the corresponding ones in Fig. \ref{boundarywalls} for the same
value of $\nu$, indicating that the global energy minimizers do not
exhibit winding. Furthermore, for $-\pi < \beta < -\pi/2$ the
solutions exhibit overshoot and a non-monotone decay to zero as
$x \to +\infty$, see Fig. \ref{windingnonmono}(c). Thus, monotonicity
of the initial data in \eqref{th0x} is not preserved under
\eqref{ELt}. Let us also mention that we tried different non-monotone
initial conditions, but did not obtain any other solutions than those
shown in Fig. \ref{windingnonmono}. However, monotone solutions with
larger winding were observed numerically for still larger values of
$\beta$. According to our numerics, it appears that edge domain wall
solutions with arbitrarily large winding are possible.

To conclude, we note that an a priori lack of monotonicity is an
obstacle for proving the precise asymptotic decay of the profiles,
using the methods of \cite{cm:non13}. A broader question of interest
is whether the one-dimensional domain wall profiles in Theorem
\ref{t:exist} are also minimizers, in some suitable sense, of the
two-dimensional micromagnetic energy in \eqref{E0}. It is well known
that magnetic domains often exhibit spatially modulated exit
structures near the material boundary \cite{hubert}, and spatially
periodic and more complicated two-dimensional edge domains have been
observed experimentally in thin films with strong in-plane crystalline
anisotropy \cite{dennis02}.

\appendix
\section{One-dimensional energy}
\label{append}

For the reader's convenience, below we present a derivation of the
one-dimensional energy in \eqref{Ebm} from the two-dimensional energy
in \eqref{E0} and then collect some rather well-known facts about the
one-dimensional fractional Sobolev norm appearing throughout our
paper.

The energy $E_0(\m)$ in \eqref{E0} with $\mathbf h$ set to zero may be
equivalently written as
\begin{align}
  \label{E0equiv}
  E_0(\m) = \frac12 \int_D \left( |\nabla \m|^2 + m_1^2 - \varphi
  \nabla \cdot \m \right) d^2 r, \qquad \varphi(\mathbf r) = -{\nu \over 4
  \pi} \int_D {\nabla \cdot \m(\mathbf r') \over |\mathbf r -
  \mathbf r'|} \, d^2 r',
\end{align}
where $\varphi$ is the stray field potential.  Since we are interested
in one-dimensional transition profiles in the vicinity of the material
edge, we assume $D$ to be an infinite strip of width $w$ oriented at
an angle $\beta \in [0, \pi/2]$ with respect to the easy axis (see
Fig. \ref{f:strip}) and consider one-dimensional magnetization
configurations. Setting $\mathbf m = \mathbf m(x)$ and
$\varphi = \varphi(x)$, where $x$ is the normal coordinate to the
strip axis defined in \eqref{x}, the energy in \eqref{E0equiv} per
unit length becomes
\begin{align}
  \label{EbmA}
  E_{\beta,w}(\mathbf m) 
  & = 
    \frac12 \int_0^w \left(
    |m_1'(x)|^2 + |m_2'(x)|^2 + m_1^2(x) \right) dx \notag \\
  & + {\nu \over 8 \pi} \lim_{L \to \infty}
    \int_{-L}^L \int_0^w \int_0^w \frac{m_\beta'(x)
    m_\beta'(y)}{\sqrt{|x-y|^2+|s|^2}} \, \, dx \, dy\,ds,
\end{align}
where $m_\beta = \mathbf e_\beta \cdot \mathbf m$, with
$\mathbf e_\beta = (\cos \beta, \sin \beta)$ and
$m_\beta(0)=m_\beta (w) =0$, and we noted that the contributions of
the magnetic dipoles to the stray field potential at distances
$|s| \gg w$ are $O(|s|^{-2})$, making the last integral in
\eqref{EbmA} convergent as $L \to \infty$. We compute
\begin{align}
  \int_{-L}^L  \frac{1}{\sqrt{|x-y|^2+|s|^2}}\,
  ds  = 2 \log \left(\frac{\sqrt{L^2+x^2}+L}{x}\right)
  = 2 \ln |x-y|^{-1} + 2 \ln (2L) + O(a^2) ,
\end{align}
with $a=\frac{|x-y|}{L}$. We know that $|x-y| \leq w$ for all
$x,y \in (0,w)$ and therefore $a \to 0$ uniformly in $x,y$ as
$L \to \infty$. Using the fact that $m_\beta(0)=m_\beta(w)=0$, we then
obtain \eqref{Ebm} for an arbitrary $\m \in H^1((0,w); \R^2)$
satisfying \eqref{mbbc}.

Now we derive several other representations of the one-dimensional
energy in \eqref{Ebm}. We can extend $m_\beta$ by zero outside the
interval $(0,w)$ and rewrite the energy as
\begin{align}
  \label{EbmA3}
  E_{\beta,w}(\mathbf m) = \frac12 \int_0^w \left( |m_1'|^2 + |m_2'|^2
  + m_1^2 \right) dx + {\nu \over 4 \pi} \int_{-\infty}^\infty
  \int_{-\infty}^\infty \ln |x - y|^{-1} \, m_\beta'(x) m_\beta'(y) \,
  dx \, dy. 
\end{align}
Next we want to rewrite the nonlocal part of the energy in Fourier
space. It is clear that $m_\beta \in H^1(\R) \cap L^1(\R)$ but
$\ln |x|^{-1}$ diverges at infinity. We introduce the following
function approximating $\ln |x|^{-1}$:
\begin{align}
  K_a (x) := e^{-a |x| } \ln |x|^{-1}, \quad a>0.
\end{align}
It is clear that $K_a(x) \in L^p(\R)$ for all $p< \infty$. Moreover,
we have as $a \to 0$
\begin{align}\label{non-con}
  {\nu \over 4 \pi} \int_{-\infty}^\infty \int_{-\infty}^\infty
  K_a(x-y) \,  m_\beta'(x) m_\beta'(y) \, dx \, dy \to  {\nu \over 4
  \pi} \int_{-\infty}^\infty \int_{-\infty}^\infty \ln |x - y|^{-1} \, 
  m_\beta'(x) m_\beta'(y) \, dx \, dy.
\end{align}
Here we used dominated convergence theorem, with the help of Young's
inequality and the fact that $m_\beta$ has compact support. Now we
can use $L^2$ Fourier transform (as defined in \eqref{F}) to obtain
\cite{lieb-loss}
\begin{align}\label{FTt}
  \int_{-\infty}^\infty \int_{-\infty}^\infty K_a(x-y) \, 
  m_\beta'(x) m_\beta'(y) \, dx \, dy =  \int_{-\infty}^\infty
  k^2 \widehat{K}_a(k) | \widehat{m}_\beta (k)|^2 \, {dk \over 2 \pi} , 
\end{align}
where \cite[Eq. (6) on p. 18]{bateman}
\begin{align}
  \widehat{K}_a(k) = {2 a \gamma + a \ln
  (k^2+a^2) + 2 k \arctan(k/a) \over k^2 + a^2},
\end{align}
and $\gamma \approx 0.5772$ is the Euler constant.  

We can split the integral in the right hand-side of \eqref{FTt} as
\begin{align}
  \int_{-\infty}^\infty  k^2 \widehat{K}_a(k) | \widehat{m}_\beta
  (k)|^2 \,  {dk \over 2 \pi}  = \int_{|k| < 1}  k^2 \widehat{K}_a(k)
  | \widehat{m}_\beta (k)|^2 \,  {dk \over 2 \pi}  + \int_{|k|>1}
  k^2 \widehat{K}_a(k) | \widehat{m}_\beta (k)|^2 \,  {dk \over 2
  \pi}.   
\end{align}
Using the facts that $\widehat{m}_\beta(k) \in L^2(\R)$,
$k \widehat{m}_\beta(k) \in L^2(\R)$ and
\begin{align}
  \widehat{K}_a(k)  \to  \frac{\pi}{|k|} \quad \hbox{ as $a \to 0$
  uniformly in $k$ for all } |k| \geq 1, \\ 
  k^2 \widehat{K}_a(k) \to  {\pi}{|k|} \quad \hbox{ as $a \to 0$
  uniformly in $k$ for all } |k| \leq 1, 
\end{align}
we obtain
\begin{align}\label{FTtc}
  \int_{-\infty}^\infty \int_{-\infty}^\infty K_a(x-y) \, 
  m_\beta'(x) m_\beta'(y) \, dx \, dy  \to \frac12
  \int_{-\infty}^\infty |k| | \widehat{m}_\beta (k)|^2 \, dk. 
\end{align}
Combining \eqref{non-con}, \eqref{FTt} and \eqref{FTtc} we arrive at
the following representation for the energy in \eqref{EbmA3}:
\begin{align}
  \label{Ebm1}
  E_{\beta,w}(\mathbf m) = \frac12 \int_0^w \left( |m_1'|^2 + |m_2'|^2
  + m_1^2 \right) dx + {\nu \over 4} \int_{-\infty}^\infty  |k| | 
  \widehat{m}_\beta (k)|^2 \, {dk \over 2 \pi},
\end{align}
which is equivalent to \eqref{Ebm2}.
%We can use the following definition of the half Laplacian operator $\left( - {d^2 \over dx^2}
%  \right)^{1/2} $ \cite{LiebLoss}
%\begin{align}
%\left( - {d^2 \over dx^2} u = 
%  \right)^{1/2}
%\end{align}
%
%
%which can be rewritten in the real space using  the operator $\left( - {d^2 \over dx^2}
%  \right)^{1/2} $ \cite{LiebLoss}
%\begin{align}
%  \label{Ebm2}
%  E_{\beta,w}(\mathbf m) = \frac12 \int_0^w \left( |m_1'|^2 + |m_2'|^2 +
%  m_1^2 \right) dx + {\nu \over 4} \int_{-\infty}^\infty
%  m_\beta \left( - {d^2 \over dx^2} \right)^{1/2}  m_\beta \, dx.
%\end{align}
%We can use the following representation of $\left( - {d^2 \over dx^2}
%  \right)^{1/2} $ \cite{}
%  \begin{align}
%  \label{halflapl}
%  \left( - {d^2 \over dx^2}
%  \right)^{1/2} u(x) = {1 \over \pi} \, 
%  \dashint_{-\infty}^\infty {u(x) - u(y) \over (x - y)^2} \, dy
%\end{align}
%to represent the energy as follows
%\begin{align}
%  \label{Ebm3}
%  E_{\beta,w}(\mathbf m) = \frac12 \int_0^w \left( |m_1'|^2 + |m_2'|^2 +
%  m_1^2 \right) dx + {\nu \over 4 \pi} \int_{-\infty}^\infty m(x) 
%  \dashint_{-\infty}^\infty {m_\beta(x) - m_\beta(y) \over (x - y)^2}
%  \, dy \, dx. 
%\end{align}
%Here and everywhere below $\dashint$ denotes the principal value
%of the integral. However, in order to avoid dealing with the principal value one can use the following form
%\cite[Proposition 3.4]{dinezza12}:

We now want to rewrite the nonlocal term in the real space involving
only $m_\beta$, but not $m'_\beta$. In order to do this we note that
\begin{align}
  \frac{1}{2 \pi} \int_{-\infty}^\infty \frac{|1- e^{ikz}|^2}{z^2}\,
  dz = \frac{1}{\pi} \int_{-\infty}^\infty \frac{1- \cos(k z)}{z^2} \,
  dz = |k|, 
\end{align}
% $|\widehat{m}_\beta (k)| \frac{|1- e^{ikz}|}{|z|} \in L^2 (\R^2)$
and obtain
\begin{align}
  \int_{-\infty}^\infty  |k| | \widehat m(k)|^2 \, dk
  & = \int_{-\infty}^\infty \left(
    \int_{-\infty}^\infty \frac{|1- e^{ikz}|^2}{z^2}\, dz \right) |
    \widehat{m}_\beta(k)|^2 \, \frac{dk}{2\pi}  \notag \\ 
  & =  \int_{-\infty}^\infty \frac{1}{z^2} \left(
    \int_{-\infty}^\infty  {|(1- e^{ikz}) \widehat{m}_\beta(k)|^2 } \,
    \frac{dk}{2\pi}  \right)\, dz \notag \\ 
  &= \int_{-\infty}^\infty  \int_{-\infty}^\infty
    \frac{ |{m_\beta} (x) - m_\beta (x+z)|^2}{z^2} \, dz \, dx \notag
  \\ 
  & = \int_{-\infty}^\infty  \int_{-\infty}^\infty
    \frac{ |{m_\beta} (x) - m_\beta (y)|^2}{|x-y|^2} \, dy \, dx. 
\end{align}
Therefore we can rewrite the original energy in the following way:
\begin{align}
  \label{EbmA5}
  E_{\beta,w}(\mathbf m) = \frac12 \int_0^w \left( |m_1'|^2 + |m_2'|^2 +
  m_1^2 \right) dx + {\nu \over 8 \pi} \int_{-\infty}^\infty
  \int_{-\infty}^\infty {(m_\beta(x) - m_\beta(y))^2 \over (x - y)^2}
  \, dx \, dy,
\end{align}
which coincides with \eqref{Ebm3}.

\paragraph*{Acknowledgements.}

The work of RGL and CBM was supported, in part, by NSF via grants
DMS-1313687 and DMS-1614948. VS would like to acknowledge support from
EPSRC grant EP/K02390X/1 and Leverhulme grant RPG-2014-226.

% \bibliography{../../nonlin,../../mura,../../stat}
% \bibliographystyle{plain}

\end{document}